\documentclass[preprint,sort&compress,review,3p]{elsarticle}
\usepackage{amssymb, amsmath}
\usepackage{array}
\usepackage{multirow}
\usepackage{moreverb}
\usepackage{graphicx,epsfig}
\usepackage{caption}
\usepackage{algorithm}
\usepackage{algorithmicx}
\usepackage{algpseudocode}
\usepackage{float}
\usepackage{stfloats}
\usepackage{float}
\usepackage{amsthm}
\usepackage{amsfonts}
\usepackage{amssymb}
\usepackage{color}
\usepackage{mathrsfs}
\usepackage{amsmath}
\usepackage{cases}

\newtheorem{theorem}{Theorem}

\newtheorem{proposition}{Proposition}
\newtheorem{lemma}{Lemma}
\newtheorem{defn}{Definition}
\newtheorem{assum}{Assumption}

\usepackage{amssymb}

\usepackage{lineno}
\journal{Springer}
\begin{document}
\begin{frontmatter}
\title{Weak Closed-loop Solvability of Linear Quadratic Stochastic Optimal Control Problems with Partial Information}

\author[address1]{Xun Li}\ead{li.xun@polyu.edu.hk}     
\author[address2]{Guangchen Wang}\ead{wguangchen@sdu.edu.cn}              
\author[address3,address4]{Jie Xiong}\ead{xiongj@sustech.edu.cn}
\author[address2]{Heng Zhang}\ead{zhangheng2828@mail.sdu.edu.cn,zhangh2828@163.com}  
\address[address1]{Department of Applied Mathematics, The Hong Kong Polytechnic University, Hong Kong, China}                                               
\address[address2]{School of Control Science and Engineering, Shandong University, Jinan 250061, China}             
\address[address3]{Department of Mathematics, Southern University of Science and Technology, Shenzhen 518055, China}        
\address[address4]{SUSTech International Center for Mathematics, Southern University of Science and Technology, Shenzhen 518055, China}  



\begin{abstract}
This paper investigates a linear quadratic stochastic optimal control (LQSOC) problem with partial information.  Firstly, by introducing two Riccati equations and a backward stochastic differential equation (BSDE), we solve this LQSOC problem under standard positive semidefinite assumptions. Secondly, by means of a perturbation approach, we study open-loop solvability of this problem when the weighting matrices in the cost functional are indefinite. Thirdly,  we investigate weak closed-loop solvability of this problem and prove the equivalence  between open-loop and weak closed-loop solvabilities. Finally, we give an  example to illustrate the way for obtaining a weak closed-loop optimal strategy.
\end{abstract}
\begin{keyword}
 Open-loop solvability; Weak closed-loop solvability; Linear quadratic  stochastic optimal control; Partial information; Riccati equation.
\end{keyword}

\end{frontmatter}

\section{Introduction}
The linear quadratic deterministic optimal control problem, which is an important yet classical problem in control area, can be traced back to \cite{Kalman,Letov,Bellman}. The system dynamics in these works are ordinary differential equations (ODEs) and the corresponding system coefficients are deterministic functions. The LQSOC problem started from Wonham \cite{Wonham} and then many researchers studied various types of LQSOC problems. See, e.g., \cite{Bismut,Bensoussan}. In the classical LQSOC theory, it is always assumed that the weighting matrices in the cost functional satisfy some positive semidefinite assumptions (see \cite[Chapter 6]{YongZhou1999}). Under these standard assumptions, the LQSOC problem has a unique optimal control, which is established by virtue of a Riccati equation. Surprisingly, Chen et al. \cite{ChenZhou1998} found that some weighting matrices in the cost functional can be indefinite and in such cases the LQSOC problem is still well-defined. This opened up research on indefinite LQSOC problems. See also \cite{ChenZhou2000,AitRamiZhou2000,AitRamiChen2001,ChenYong2001} for some follow-up works on this subject. 

Meanwhile, motivated by the fact that we may not observe full information at any given time, the stochastic optimal control problem with partial information is of great interest for many researchers and engineers. For example, Baghery and Øksendal \cite{Baghery} proved a maximum principle for a stochastic control problem with jumps and partial information.  Huang et al. \cite{HuangWangXiong2009} obtained maximum principles for some partial information LQSOC problems governed by BSDEs. Meng \cite{Meng2009} gave a maximum principle for a partial information stochastic problem of a fully coupled forward-backward stochatic differential equation (FBSDE). Wang and Xiao \cite{WangXiao2015} studied partial information stochastic optimal control problems of a fully coupled FBSDE for both  finite-horizon and  infinite-horizon cases. They derived two sufficient conditions for optimality of these problems. For more theoretical results and practical applications of stochastic optimal control problems with partial information, see also book \cite{WangWuXiongBook2018} and its references.
 
In 2014, Sun and Yong \cite{SunYong2014_Game} considered the open-loop and closed-loop saddle points for  a linear quadratic stochatis zero-sum differential game. Then, Sun et al. \cite{SunLiYong2016_SLQ} further investigated the open-loop and closed-loop solvabilities of an LQSOC problem by means of the uniform convexity condition of the cost functional. They found that the closed-loop solvability of a finite-horizon LQSOC problem implies the open-loop solvability of this problem, but not vice versa. Thus, when some weighting matrices in the cost functional are indefinite, there exist some LQSOC problems that are not closed-loop solvable but open-loop solvable. For those problems, their open-loop optimal strategies may not admit state feedback representations. As an extension of the closed-loop solvability for LQSOC problems in \cite{SunLiYong2016_SLQ}, Wang et al. \cite{WangSunYong2019} introduced the notion of weak closed-loop solvability. They showed that the open-loop and weak closed-loop solvabilities of LQSOC problems are equivalent. Along this line, the open-loop, closed-loop and weak closed-loop solvabilities for mean-field LQSOC problems are studied in \cite{Sun_mean-field_open,LiSunYong_mean-field_closed,SunWang_weak_closed}, respectively. For LQSOC problems of Markovian regime
switching systems, Zhang et al. \cite{ZhangLiXiong2021} analyzed their open-loop and closed-loop solvabilities and Wen et al. \cite{WenLiXiong2021} considered their weak closed-loop solvabilities.

Motivated by the above literature, it is natural to ask: \textit{How are the  open-loop, closed-loop and weak closed-loop solvabilities defined for LQSOC problems with partial information? What is the relationship among these solvabilities?} The objective of this paper is to answer this question by mainly investigating open-loop and weak closed-loop solvabilities. The closed-loop solvability has been studied in a forthcoming paper.

Let us begin with $(\Omega,\mathcal{F},\mathbb{P})$ being a complete probability space, on which two independent standard one-dimensional Brownian motions $\{W_1(t), W_2(t)\}_{0\leq t<\infty}$ are defined. Let $\mathbb{F}=\{\mathcal{F}_t\}_{t\geq0}$ be the natural filtration of $\{W_1(\cdot), W_2(\cdot)\}$ augmented by all $\mathbb{P}-$null sets in $\mathcal{F}$. Let $\mathbb{G}=\{\mathcal{G}_t\}_{t\geq0}$ be the natural filtration of $ W_2(\cdot)$ augmented by all $\mathbb{P}-$null sets in $\mathcal{G}$. In this paper, $\mathcal{G}_t$ stands for the information we can observe at time $t$. That is,  at time $t$, we can only observe partial information of $\mathcal{F}_t$ .

Consider the following stochastic differential equation (SDE) on a finite horizon $[s,T]$
\begin{equation}
	\label{system}
	\begin{cases}
		\begin{split}
			dx(t)= \,\,&\big[A(t)x(t)+B(t)u(t)+b(t)\big]dt+\big[C_1(t)x(t)+D_1(t)u(t)+\sigma_1(t)\big]dW_1(t)\\
			&+\big[C_2(t)x(t)+D_2(t)u(t)+\sigma_2(t)\big]dW_2(t),\quad t\in[s,T],\\
		\end{split}\\
		x(s)=x_0,
	\end{cases}
\end{equation}
where $x(\cdot)\in\mathbb{R}^n$ represents the state, $u(\cdot)\in\mathbb{R}^m$ represents the control and $x_0\in\mathbb{R}^n$ is the initial state. The coefficients $A(\cdot)$, $B(\cdot)$, $C_1(\cdot)$, $C_2(\cdot)$, $D_1(\cdot)$, $D_2(\cdot)$ are deterministic functions of proper sizes and $b(\cdot)$, $\sigma_1(\cdot)$, $\sigma_2(\cdot)$ are $\mathbb{G}-$progressively measurable processes of proper sizes. Given any $s\in[0,T)$, the admissible control set of this paper is
\begin{equation}\begin{split}
		\mathcal{U}_\mathbb{G}[s,T]\triangleq \Big\{u:[s,T]\times\Omega\rightarrow\mathbb{R}^m\big|u(\cdot)\  \text{is}\ \mathbb{G}-\text{progressively measurable, and}\  \mathbb{E}\int_{s}^{T}|u(t)|^2dt<\infty\Big\}.
	\end{split}
\end{equation}

The cost functional considered in this paper is

\begin{equation}\label{cost}
	\begin{split}
		J(s,x_0;u(\cdot))\triangleq\ &\mathbb{E}\int_{s }^{T}\left[\left<\begin{pmatrix}
			Q(t) & S(t)^\top\\
			S(t) & R(t)\\
		\end{pmatrix}\begin{pmatrix}
			x(t)\\
			u(t)\\
		\end{pmatrix},\begin{pmatrix}
			x(t)\\
			u(t)\\
		\end{pmatrix}\right>+2\left<\begin{pmatrix}
			q(t)\\
			\rho(t)\\
		\end{pmatrix},\begin{pmatrix}
			x(t)\\
			u(t)\\
		\end{pmatrix}\right>\right]dt\\
		&+\mathbb{E}\left[\left<G x(T),x(T)\right> +2\left<g ,x(T)\right>\right],
	\end{split}
\end{equation}
where $Q(\cdot)$, $R(\cdot)$ are symmetric and deterministic functions of proper sizes, $S(\cdot)$ is a deterministic function of proper size, $q(\cdot)$, $\rho(\cdot)$ are $\mathbb{G}-$progressively measurable processes of proper sizes and $g$ is a $\mathcal{G}_T-$measurable random variable.

The LQSOC problem with partial information is as follows.

\noindent\textbf{Problem (P)} Given any $(s,x_0)\in[0,T)\times\mathbb{R}^n$, search for an optimal control $u^*(\cdot)\in\mathcal{U}_\mathbb{G}[s,T]$ such that 
\begin{equation}
	V(s,x_0)\triangleq J(s,x_0;u^*(\cdot))\leq J(s,x_0;u(\cdot)),\quad \forall u(\cdot)\in\mathcal{U}_\mathbb{G}[s,T].
\end{equation}

For simplicity, when $b(\cdot)$, $\sigma_1(\cdot)$, $\sigma_2(\cdot)$, $q(\cdot)$, $\rho(\cdot)$, $g=0$, we denote this problem by Problem (P)$^0$, the corresponding cost functional by $J^0(s,x_0;u(\cdot))$ and the corresponding value function by $V^0(s,x_0)$.

 Noteworthily, the fact that we can only observe partial information brings some difficulties to investigate Problem (P). Firstly, how to give the definitions of open-loop, closed-loop and weak closed-loop solvabilities  when we only know partial information $\mathbb{G}$?  Secondly, different from the full information case in \cite{SunLiYong2016_SLQ,WangSunYong2019},  what do the Riccati equation and the corresponding BSDE caused by the inhomogeneous terms look like? Thirdly, how to show the equivalence between the open-loop and weak closed-loop solvabilities of Problem (P)? Inspired by the methodologies in \cite{WangWuXiongBook2018,SunLiYong2016_SLQ,WangSunYong2019}, we overcome these difficulties. The main contributions of this paper are as follows.
\begin{itemize}
	\item By virtue of the filtering equation of the system dynamic, we give the definition of open-loop, closed-loop and weak closed-loop solvabilities for LQSOC problems with partial information. In these definitions, the control does not belong to the full information set $\mathbb{F}$, but  to the partial information set $\mathbb{G}$.
	\item With the help of two Riccati equations and a BSDE, we study the open-loop solvability of Problem (P) by analyzing the cost functional and by adopting a perturbation method similar to Wang et al. \cite{WangSunYong2019}, respectively.   The former method  show that the open-loop solvability can be obtained by introducing an FBSDE and imposing some conditions on the cost functional (see Theorem \ref{theorem_openLOOP_FBSDE}). The latter method characterize  the open-loop solvability by studying a sequence of perturbed versions of Problem (P) (see Theorem \ref{theorem_4}).
	\item Based on the perturbed versions of Problem (P), we prove that the open-loop and weak closed-loop solvabilities of Problem (P) are equivalent. This means that, if Problem (P) is open-loop solvable, then we can find a weak closed-loop strategy and its outcome is an open-loop optimal strategy. We also give an illustrative example to show the way for finding the optimal weak closed-loop strategy. Distinguished with the results in  \cite{SunLiYong2016_SLQ,WangSunYong2019}, we further develop their results to the case with partial information. This may be more in line with the practical application background and may be helpful for industrial practitioners.
\end{itemize}

This paper is arranged as follows. Section \ref{sec 1} presents some preliminaries.  Section \ref{sec 3} devotes to the open-loop solvability of Problem (P). Section \ref{sec 4} proves the equivalence between open-loop and weak closed-loop solvabilities. Section \ref{sec 5} gives an example to illustrate the theoretical results. Section \ref{sec 6} concludes this paper and collects some promising future works.

\section{Preliminaries}\label{sec 1}
In this paper, we denote respectively by $tr(N)$, $N^{-1}$ and $N^\dagger$ the trace,  the inverse and  the Moore-Penrose pseudoinverse of matrix $N$.  $\top$ denotes the transpose symbol of a matrix or a vector.  $\mathbb{R}^{n}$ represents the  Euclidean space of dimension $n$  and  $\mathbb{R}^{m\times n}$ is the space of all $m\times n$ constant matrices with its inner product being $\left< N,M\right>=tr(N^\top M)$ and its norm being $|N|=\sqrt{tr(N^\top N)}$.  We also use $\big<\cdot,\cdot\big>$ to denote inner products in other Hilbert spaces, if there is no ambiguity. $\mathbb{S}^n$ and $\mathbb{S}^n_+$ stand for the spaces of all $n\times n$ constant symmetric matrices and positive semidefinite matrices, respectively. $\mathbb{I}_m$ denotes the space of $m\times m$ identity matrices. For any matrices $N$, $M\in\mathbb{S}^m$, if $N-M$ is positive definite (respectively, positive semidefinite), we write $N>M$ (respectively, $N\geq M$). Given any $s\in[0,T)$ and Euclidean space $\mathbb{X}$, we define the following spaces:  
\begin{equation*}
	\begin{split}
		&C([s,T];\mathbb{X})\triangleq\left\{f:[s,T]\rightarrow\mathbb{X}\ \big|f\ \text{is continuous}\right\},\\
		&L^p([s,T];\mathbb{X})\triangleq\left\{f:[s,T]\rightarrow\mathbb{X}\ \big|f\ \text{is}\ p\text{th}\ (1\leq p<\infty)\  \text{power Lebesgue integrable}\right\},\\
		&L^\infty([s,T];\mathbb{X})\triangleq\left\{f:[s,T]\rightarrow\mathbb{X}\ \big|f\ \text{is essentially bounded measurable function}\right\},\\
		&L^2_{\mathcal{G}_T}(\Omega;\mathbb{X})\triangleq\left\{f:\Omega\rightarrow\mathbb{X}\ \big|f\ \text{is}\ \mathcal{G}_T-\text{measurable},\ \mathbb{E}|f|^2<\infty\right\},\\
		&L^2_{\mathbb{G}}([s,T];\mathbb{X})\triangleq\left\{f:[s,T]\times\Omega\rightarrow\mathbb{X}\ \big|f\ \text{is}\ \mathbb{G}-\text{progressively measurable,}\ \mathbb{E}\int_{s}^{T}|f(t)|^2dt<\infty\right\},\\
		&L^2_{\mathbb{G}}(\Omega;L^1([s,T];\mathbb{X}))\triangleq\left\{f:[s,T]\times\Omega\rightarrow\mathbb{X}\ \big|f\ \text{is}\ \mathbb{G}-\text{progressively measurable,}\ \mathbb{E}\left[\int_{s}^{T}|f(t)|dt\right]^2<\infty\right\},\\
		&L^2_{\mathbb{G}}(\Omega;C([s,T];\mathbb{X}))\triangleq\left\{f:[s,T]\times\Omega\rightarrow\mathbb{X}\ \big|f\ \text{is}\ \mathbb{G}-\text{adapted, continuous,}\ \mathbb{E}\left[\sup\limits_{s\leq t\leq T}|f(t)|^2\right]<\infty\right\}.\\
	\end{split}
\end{equation*}

Throughout this paper, we introduce the following assumptions:
\begin{assum}\label{assum_system}
	The parameters in system (\ref{system}) satisfy $A(\cdot), C_1(\cdot), C_2(\cdot)\in L^\infty(0,T;\mathbb{R}^{n\times n})$, $B(\cdot), D_1(\cdot), D_2(\cdot)\in L^\infty(0,T;\mathbb{R}^{n\times m})$,  $b(\cdot)\in L^2_{\mathbb{G}}(\Omega;L^1([0,T];\mathbb{R}^n))$, $\sigma_1(\cdot)\in L^2_{\mathbb{G}}([0,T];\mathbb{R}^n)$, $\sigma_2(\cdot)\in L^2_{\mathbb{G}}([0,T];\mathbb{R}^n)$.
\end{assum} 

\begin{assum}\label{assum_cost}
	The parameters in cost functional (\ref{cost}) satisfy $Q(\cdot)\in L^\infty(0,T;\mathbb{S}^{n})$, $R(\cdot)\in L^\infty(0,T;\mathbb{S}^{ m})$, $S(\cdot)\in L^\infty(0,T;\mathbb{R}^{m\times n})$, $q(\cdot)\in L^2_{\mathbb{G}}(\Omega;L^1([0,T];\mathbb{R}^n))$, $\rho(\cdot)\in L^2_{\mathbb{G}}([0,T];\mathbb{R}^m)$, $g\in L^2_{\mathcal{G}_T}(\Omega;\mathbb{R}^n)$, $G\in\mathbb{S}^n$.
\end{assum} 

The next lemma presents the well-posedness of system (\ref{system}). Since its proof is similar to \cite[Proposition 2.1]{SunYong2014_Game}, herein we omit it.

\begin{lemma}\label{lemma 1}
	Let Assumption \ref{assum_system} hold. Given any $(s,x_0)\in[0,T)\times \mathbb{R}^n$ and $u(\cdot)\in\mathcal{U}_\mathbb{G}[s,T]$, system (\ref{system}) admits a unique strong solution. Furthermore, there exists a real number $L>0$, which is independent of $u(\cdot)$ and $(s,x_0)$, such that 
	\begin{equation}
		\begin{split}
			\mathbb{E}\left[\sup\limits_{s\leq t\leq T}|x(t)|^2\right]\leq L\mathbb{E}\left[|x_0|^2+\left(\int_{s}^{T}|b(t)|dt\right)^2+\int_{t}^{T}\left(|\sigma_1(t)|^2+|\sigma_2(t)|^2\right)dt+\int_{t}^{T}|u(t)|^2dt\right].
		\end{split}
	\end{equation}
\end{lemma}

Further, we give the definitions of open-loop, closed-loop and weak closed-loop solvabilities for Problem (P). For notation convenience, we suppress the argument $t$ when there is no confusion.

\begin{defn}\label{defnition 1}
(i) 	Problem (P) is called (uniquely) open-loop sovlvable at $(s,x_0)\in[0,T)\times \mathbb{R}^n$ if there is a (unique) control $u^*(\cdot)\in\mathcal{U}_\mathbb{G}[s,T]$ such that
\begin{equation*}
	J(s,x_0,u^*(\cdot))\leq J(s,x_0;u(\cdot)),\quad \forall u(\cdot)\in\mathcal{U}_\mathbb{G}[s,T].
\end{equation*}
In this case, $u^*(\cdot)$ is said to be an open-loop optimal control.

(ii) Problem (P) is called (uniquely) open-loop sovlvable if it is (uniquely) open-loop sovlvable for any $(s,x_0)\in[0,T)\times \mathbb{R}^n$.
\end{defn}

\begin{defn}
	Suppose that $\varTheta: [s,T]\rightarrow\mathbb{R}^{m\times n}$ is  a deterministic function and $\varLambda:[s,T]\times\Omega\rightarrow\mathbb{R}^{m}$ is a $\mathbb{G}-$progressively measurable process. That is
	\begin{equation*}
		\int_{s}^{T}|\varTheta(t)|^2dt<\infty, \quad \mathbb{E}\int_{s}^{T}|\varLambda(t)|^2dt<\infty.
	\end{equation*}
	
(i) $(\varTheta(\cdot),\varLambda(\cdot))$ is called a closed-loop strategy on $[s,T]$ if $\varTheta(\cdot)\in L^2(s,T;\mathbb{R}^{m\times n})$ and $\varLambda(\cdot)\in L^2_\mathbb{G}(s,T;\mathbb{R}^{m})$. The collection of closed-loop strategies on $[s,T]$ is defined as  $\mathcal{Q}[s,T]$.

(ii)  A closed-loop strategy $(\varTheta^*(\cdot),\varLambda^*(\cdot))\in\mathcal{Q}[s,T]$ is called optimal on $[s,T]$ if 
\begin{equation*}
	J(s,x_0,\varTheta^*(\cdot)\widehat{x}^*(\cdot)+\varLambda^*(\cdot))\leq J(s,x_0;\varTheta(\cdot)\widehat{x}(\cdot)+\varLambda(\cdot)),\quad \forall x_0\in\mathbb{R}^n, \quad \forall (\varTheta(\cdot),\varLambda(\cdot))\in\mathcal{Q}[s,T],
\end{equation*}
where $\widehat{x}^*(\cdot)$ and $\widehat{x}(\cdot)$ evolve according to 
\begin{equation*}
	\begin{cases}
		\begin{split}
			d\widehat{x}^*(t)= \big[(A+B\varTheta^*)\widehat{x}^*(t)+B\varLambda^*+b\big]dt
			+\big[(C_2+D_2\varTheta^*)\widehat{x}^*(t)+D_2\varLambda^*+\sigma_2\big]dW_2(t),\quad t\in[s,T],
		\end{split}\\
		\widehat{x}^*(s)=x_0,\\
	\end{cases}
\end{equation*}
and 
\begin{equation*}
	\begin{cases}
		\begin{split}
			d\widehat{x}(t)= \big[(A+B\varTheta)\widehat{x}(t)+B\varLambda+b\big]dt
			+\big[(C_2+D_2\varTheta)\widehat{x}(t)+D_2\varLambda+\sigma_2\big]dW_2(t),\quad t\in[s,T],
		\end{split}\\
		\widehat{x}^*(s)=x_0.\\
	\end{cases}
\end{equation*}

(iii) Problem (P) is called (uniquely) closed-loop solvable if there (uniquely) exists a closed-loop strategy on $[s,T]$ for any $s\in[0,T)$.
\end{defn}

\begin{defn}\label{definition 3}
	Assume that $\varTheta: [s,T]\rightarrow\mathbb{R}^{m\times n}$ is a locally square-integrable deterministic function and $\varLambda:[s,T]\times\Omega\rightarrow\mathbb{R}^{m}$ is a locally square-integrable $\mathbb{G}-$progressively measurable process. That is, given any $T^{'}\in[s,T)$, we have
	\begin{equation*}
		\int_{s}^{T^{'}}|\varTheta(t)|^2dt<\infty, \quad \mathbb{E}\int_{s}^{T^{'}}|\varLambda(t)|^2dt<\infty.
	\end{equation*}

(i) $(\varTheta(\cdot),\varLambda(\cdot))$ is called a weak closed-loop strategy on $[s,T]$ if, given any $x_0\in\mathbb{R}^n$, the outcome $u(\cdot)=\varTheta(\cdot)\widehat{x}(\cdot)+\varLambda(\cdot)$ of $(\varTheta(\cdot),\varLambda(\cdot))$ belongs to $\mathcal{U}_\mathbb{G}[s,T]$, where $\widehat{x}(\cdot)$ evolves according to
\begin{equation}\label{eq6}
	\begin{cases}
		\begin{split}
			d\widehat{x}(t)= \big[(A+B\varTheta)\widehat{x}(t)+B\varLambda+b\big]dt
			+\big[(C_2+D_2\varTheta)\widehat{x}(t)+D_2\varLambda+\sigma_2\big]dW_2(t),\quad t\in[s,T],
		\end{split}\\
		\widehat{x}^*(s)=x_0.\\
	\end{cases}
\end{equation}
The collection of weak closed-loop strategies on $[s,T]$ is defined as  $\mathcal{Q}_w[s,T]$.

(ii)  A weak closed-loop strategy $(\varTheta^*(\cdot),\varLambda^*(\cdot))\in\mathcal{Q}[s,T]$ is called optimal on $[s,T]$ if 
\begin{equation*}
	J(s,x_0,\varTheta^*(\cdot)\widehat{x}^*(\cdot)+\varLambda^*(\cdot))\leq J(s,x_0;\varTheta(\cdot)\widehat{x}(\cdot)+\varLambda(\cdot)),\quad \forall x_0\in\mathbb{R}^n, \quad \forall (\varTheta(\cdot),\varLambda(\cdot))\in\mathcal{Q}_w[s,T],
\end{equation*}
where $\widehat{x}^*(\cdot)$ and $\widehat{x}(\cdot)$ are solutions of (\ref{eq6}) with $(\varTheta^*(\cdot),\varLambda^*(\cdot))$ and $(\varTheta(\cdot),\varLambda(\cdot))$, respectively. 

(iii) Problem (P) is called (uniquely) weak closed-loop solvable if there (uniquely) exists a weak closed-loop strategy on $[s,T]$ for any $s\in[0,T)$.
\end{defn}

In order to solve Problem (P), we introduce two Riccati equations and a backward stochastic differential equation (BSDE):
\begin{equation}\label{Riccati_1}
	\begin{split}
		\begin{cases}
			\dot{P}^1+A^\top P^1 +P^1A+C_1^\top P^1C_1+C_2^\top P^1C_2+Q=0,\\
			P^1(T)=G,
		\end{cases}
	\end{split}
\end{equation}

\begin{equation}\label{Riccati_2}
		\begin{cases}
				\begin{split}
			\dot{P}^2&+A^\top P^2 +P^2A+C_1^\top P^1C_1+C_2^\top P^2C_2\\
			&-(B^\top P^2+D_1^\top P^1C_1+D_2^\top P^2C_2+S)^\top(R+D_1^\top P^1D_1+D_2^\top P^2D_2)^\dagger\\
			&\times(B^\top P^2+D_1^\top P^1C_1+D_2^\top P^2C_2+S)+Q=0,\\	
		\end{split}\\
			P^2(T)=G,\\
		\end{cases}
\end{equation}

\begin{equation}\label{BSDE}
	\begin{cases}
		\begin{split}
			d\alpha(t)=&-\Big[\big(A-B(R+D_1^\top P^1D_1+D_2^\top P^2D_2)^\dagger(B^\top P^2+D_1^\top P^1C_1+D_2^\top P^2C_2+S)\big)^\top\alpha(t)\\
			&+\big(C_2-D_2(R+D_1^\top P^1D_1+D_2^\top P^2D_2)^\dagger(B^\top P^2+D_1^\top P^1C_1+D_2^\top P^2C_2+S)\big)^\top\beta(t)\\
			&+\big(C_1-D_1(R+D_1^\top P^1D_1+D_2^\top P^2D_2)^\dagger(B^\top P^2+D_1^\top P^1C_1+D_2^\top P^2C_2+S)\big)^\top P^1\sigma_1\\
			&+\big(C_2-D_2(R+D_1^\top P^1D_1+D_2^\top P^2D_2)^\dagger(B^\top P^2+D_1^\top P^1C_1+D_2^\top P^2C_2+S)\big)^\top P^2\sigma_2\\
			&-\big((R+D_1^\top P^1D_1+D_2^\top P^2D_2)^\dagger(B^\top P^2+D_1^\top P^1C_1+D_2^\top P^2C_2+S)\big)^\top\rho+P^2b+q\Big]dt\\
			&+\beta(t) dW_2(t),\\
		\end{split}\\
		\alpha(T)=g.\\
	\end{cases}
\end{equation}

Then, we present the next lemma to reveal some properties of Riccati equation (\ref{Riccati_1}).
\begin{lemma}\label{Riccati_1_Lemma}
	Suppose that Assumptions \ref{assum_system}-\ref{assum_cost} hold, then equation (\ref{Riccati_1}) admits a unique solution $P^1(\cdot)\in C([0,T];\mathbb{S}^n)$.  If, in addition, $G\geq 0$ and $Q(t)\geq 0$, $\forall t\in[0,T]$, then $P^1(t)\geq 0$, $\forall t\in[0,T]$.
\end{lemma}
\begin{proof}
	The proof is a multidimensional Brownian motion version of \cite[Lemma 7.3]{YongZhou1999}. So we omit it for space consideration.
\end{proof}

\section{Open-loop solvabilities}\label{sec 3}
\subsection{Positive semidefinite cases}
In order to compare with indefinite cases, this subsection solves Problem (P) under the following classical positive semidefinite assumptions:

\begin{assum}\label{assum_positive definite matrix}
	The weighting matrices in cost functional (\ref{cost}) satisfy $G\geq 0$, $Q(t)-S(t)^\top R(t)^{-1}S(t)\geq 0$ and $R(t)\geq\delta\mathbb{I}_m$, $\forall t\in[0,T]$, where $\delta>0$ is a constant.
\end{assum}

By virtue of Assumption \ref{assum_positive definite matrix}, we have the following lemma.

\begin{lemma}\label{Riccati_2_lemma}
	Let Assumptions \ref{assum_system}-\ref{assum_positive definite matrix} hold, then Riccati equation (\ref{Riccati_2}) admits a unique solution $P^2(\cdot)\in C([0,T];\mathbb{S}^n_+)$.
\end{lemma}

\begin{proof}
	By combining Lemma \ref{Riccati_1_Lemma} with \cite[Theorem 7.2]{YongZhou1999}, the statements of this lemma is obtained. This completes the proof.
\end{proof}

Based on the above results, now we give the optimal control of Problem (P) under Assumption \ref{assum_positive definite matrix}.
\begin{theorem}\label{optimal proof}
	Let Assumptions \ref{assum_system}-\ref{assum_positive definite matrix} hold. The optimal control of Problem (P) is 
	\begin{equation}\label{optimal control_definite}
	\begin{split}
	u^*(t)=\varTheta^*(t)\widehat{x}^*(t)+\varLambda^*(t),
	\end{split}
	\end{equation}
	where 
	\begin{equation*}
	\begin{split}
	\varTheta^*\triangleq&-(R+D_1^\top P^1D_1+D_2^\top P^2D_2)^{-1}(B^\top P^2+D_1^\top P^1C_1+D_2^\top P^2C_2+S),\\
	\varLambda^*\triangleq&-(R+D_1^\top P^1D_1+D_2^\top P^2D_2)^{-1}(B^\top\alpha+D_2^\top\beta+D_1^\top P^1\sigma_1+D_2^\top P^2\sigma_2+\rho),\\
	\end{split}
	\end{equation*} $P^1(\cdot)$, $P^2(\cdot)$ are solutions to Riccati equations (\ref{Riccati_1}) and (\ref{Riccati_2}), $(\alpha(\cdot),\beta(\cdot))$ is the adapted solution of BSDE (\ref{BSDE}) and $\widehat{x}(\cdot)$ evolves according to
	\begin{equation*}
	\label{optimal filter}
	\begin{cases}
	\begin{split}
	d\widehat{x}^*(t)= \big[A\widehat{x}^*(t)+Bu^*(t)+b\big]dt
	+\big[C_2\widehat{x}^*(t)+D_2u^*(t)+\sigma_2\big]dW_2(t),\quad t\in[s,T],
	\end{split}\\
	\widehat{x}^*(s)=x_0.\\
	\end{cases}
	\end{equation*}
\end{theorem}

\begin{proof}
	Under Assumptions \ref{assum_system}-\ref{assum_positive definite matrix}, it follows from Lemmas \ref{Riccati_1_Lemma}-\ref{Riccati_2_lemma} that
	\begin{equation*}
	(R+D_1^\top P^1D_1+D_2^\top P^2D_2)^\dagger=(R+D_1^\top P^1D_1+D_2^\top P^2D_2)^{-1}
	\end{equation*}
	Thus, the pseudoinverse symbols in equations (\ref{Riccati_2}) and (\ref{BSDE})  can be written as inverse symbols.
	
	Applying It\^o's formula to $d\big(x^\top(t) P^1 x(t)\big)$, $d\big(\widehat{x}^\top(t) (P^2(t)-P^1(t)) \widehat{x}(t)\big)$ and $d\big(x(t)^\top\alpha(t)^\top\big)$, where $\widehat{x}(\cdot)$ is the state of
	\begin{equation*}
	\begin{cases}
	\begin{split}
	d\widehat{x}(t)= \big[A\widehat{x}(t)+Bu(t)+b\big]dt
	+\big[C_2\widehat{x}(t)+D_2u(t)+\sigma_2\big]dW_2(t),\quad t\in[s,T],
	\end{split}\\
	\widehat{x}(s)=x_0,\\
	\end{cases}
	\end{equation*}
	with any $u(\cdot)\in\mathcal{U}_\mathbb{G}[s,T]$, then we  obtain
	\begin{equation}\label{eq14}
	\begin{split}
	&\mathbb{E}\big[x(T)^\top Gx(T)-x_0^\top P^1(s)x_0\big]\\
	=\ &\mathbb{E}\int_{s}^{T}\Big\{[Ax+Bu+b]^\top P^1x+x^\top \dot{P}^1 x
	+x^\top P^1 [Ax+Bu+b]\\
	&+[C_1x+D_1u+\sigma_1]^\top P^1[C_1x+D_1u+\sigma_1]
	+[C_2x+D_2u+\sigma_2]^\top P^1[C_2x+D_2u+\sigma_2]\Big\}dt,\\
	\end{split}
	\end{equation}
	\begin{equation}\label{eq15}
	\begin{split}
	&-\mathbb{E}\big[x_0^\top (P^2(s)-P^1(s))x_0\big]\\
	=\ &\mathbb{E}\int_{s}^{T}\Big\{[A\widehat{x}+Bu+b]^\top (P^2-P^1)\widehat{x}+\widehat{x}^\top (\dot{P}^2-\dot{P}^1) \widehat{x}
	+\widehat{x}^\top (P^2-P^1) [A\widehat{x}+Bu+b]\\
	&+[C_2\widehat{x}+D_2u+\sigma_2]^\top (P^2-P^1)[C_2\widehat{x}+D_2u+\sigma_2]\Big\}dt,\\
	\end{split}
	\end{equation}
	and
	\begin{equation}\label{eq16}
	\begin{split}
	&\mathbb{E}\big[x(T)^\top g-x_0^\top \alpha(s)\big]\\
	=\ &\mathbb{E}\int_{s}^{T}\Big\{[Ax+Bu+b]^\top \alpha-x^\top\big[(A+B\varTheta^*)^\top\alpha+(C_2+D_2\varTheta^*)^\top\beta+(C_1+D_1\varTheta^*)^\top P^1\sigma_1\\
	&+(C_2+D_2\varTheta^*)^\top P^2\sigma_2+\varTheta^*\rho+P^2b+q\big]
	+[C_2x+D_2u+\sigma_2]^\top \beta\Big\}dt.\\
	\end{split}
	\end{equation}
	
	Combining (\ref{eq14}), (\ref{eq15}), (\ref{eq16}) with the definition of cost functional (\ref{cost}), we have
	\begin{equation*}
	\begin{split}
	J(s,x_0;u(\cdot))=\ &\mathbb{E}\int_{s}^{T}\big\{x^\top Qx+u^\top Ru+2x^\top S^\top u+2x^\top q+2u^\top\rho\big\}dt+\mathbb{E}\big[x(T)^\top Gx(T)+2x(T)^\top g\big]\\
	=\ &\mathbb{E}\int_{s}^{T}\Big\{x^\top Qx+u^\top Ru+2x^\top S^\top u+2x^\top q+2u^\top\rho+[Ax+Bu+b]^\top P^1x+x^\top \dot{P}^1 x\\
	&+x^\top P^1 [Ax+Bu+b]
	+[C_1x+D_1u+\sigma_1]^\top P^1[C_1x+D_1u+\sigma_1]\\
	&+[C_2x+D_2u+\sigma_2]^\top P^1[C_2x+D_2u+\sigma_2]+[A\widehat{x}+Bu+b]^\top (P^2-P^1)\widehat{x}\\
	&+\widehat{x}^\top (\dot{P}^2-\dot{P}^1) \widehat{x}
	+\widehat{x}^\top (P^2-P^1) [A\widehat{x}+Bu+b]\\
	&+[C_2\widehat{x}+D_2u+\sigma_2]^\top (P^2-P^1)[C_2\widehat{x}+D_2u+\sigma_2]+2[Ax+Bu+b]^\top \alpha\\
	&-2x^\top\big[(A+B\varTheta^*)^\top\alpha+(C_2+D_2\varTheta^*)^\top\beta+(C_1+D_1\varTheta^*)^\top P^1\sigma_1\\
	&+(C_2+D_2\varTheta^*)^\top P^2\sigma_2+\varTheta^*\rho+P^2b+q\big]
	+2[C_2x+D_2u+\sigma_2]^\top \beta
	\Big\}dt\\
	&+\mathbb{E}\big[x_0^\top P^2(s)x_0+2x_0^\top \alpha\big]\\
	=\ &\mathbb{E}\int_{s}^{T}\Big\{\big[u-\varTheta^* \widehat{x}-\varLambda^*\big]^\top(R+D_1^\top P^1D_1+D_2^\top P^2D_2)\big[u-\varTheta^* \widehat{x}-\varLambda^*\big]\\
	&+\sigma_1^\top P^1\sigma_1+\sigma_2^\top P^2\sigma_2+2\sigma_2^\top\beta+2b^\top\alpha-(\varLambda^*)^\top(R+D_1^\top P^1D_1+D_2^\top P^2D_2)\varLambda^*\Big\}dt\\
	&+\mathbb{E}\big[x_0^\top P^2(s)x_0+2x_0^\top \alpha\big].\\
	\end{split}
	\end{equation*}
	The above equation shows that the optimal control is (\ref{optimal control_definite}) and thus the proof is completed.
\end{proof}

\subsection{Indefinite cases}
In the sequel, we do not impose any nonnegativeness/ positive-semidefiniteness conditions on the weighting matrices in cost functional (\ref{cost}). First, we show the equivalence between the open-loop solvability and the corresponding FBSDE.

\begin{theorem}\label{theorem_openLOOP_FBSDE}
	Assume that Assumptions \ref{assum_system}, \ref{assum_cost} hold. Given any $(s, x_0)\in[0,T)\times\mathbb{R}^n$ and $\lambda\in\mathbb{R}^1$,  then we have
	
	(i)  For any $u(\cdot)$, $\nu(\cdot)\in\mathcal{U}_\mathbb{G}[s,T]$, the cost functional satisfies
	\begin{equation*}
		\begin{split}
			J(s,x_0;u(\cdot)+\lambda \nu(\cdot))=J(s,x_0;u(\cdot))+\lambda^2J^0(s,0;\nu(\cdot
			))+\lambda\mathbb{E}\int_{s}^{T}\nu(t)^\top \mathcal{D}J(s,x_0;u(t))(t)dt,
		\end{split}
	\end{equation*}
	where
	\begin{equation}\label{eq17}
		\begin{split}
			\mathcal{D}J(s,x_0;u(\cdot))(t)=2\Big[B(t)^\top Y(t)+D_1(t)^\top Z_1(t)+D_2(t)^\top Z_2(t)+S(t)x(t)+R(t)u(t)+\rho(t)\Big],
		\end{split}
	\end{equation}
	and $x(\cdot)$, $Y(\cdot)$, $Z(\cdot)$ are solutions of the following FBSDE
	\begin{equation}\label{eq18}
		\begin{cases}
			\begin{split}
				dx(t)= \,\,&\big[A(t)x(t)+B(t)u(t)+b(t)\big]dt+\big[C_1(t)x(t)+D_1(t)u(t)+\sigma_1(t)\big]dW_1(t)\\
				&+\big[C_2(t)x(t)+D_2(t)u(t)+\sigma_2(t)\big]dW_2(t),\quad t\in[s,T],\\
			\end{split}\\
			\begin{split}
				dY(t)= \,\,&-\big[A(t)^\top Y(t)+C_1(t)^\top Z_1(t)+C_2(t)^\top Z_2(t)+Q(t)x(t)+S(t)^\top u(t)+q(t)\big]dt\\
				&+Z_1(t)dW_1(t)+Z_2(t)dW_2(t),\quad t\in[s,T],\\
			\end{split}\\
			x(s)=0, \quad Y(T)=Gx(T)+g.\\
		\end{cases}
	\end{equation}

(ii) A control $u(\cdot)$ is an open-loop optimal strategy of Problem (P) if any only if $J^0(s,0;\nu(\cdot))\geq 0, \forall \nu(\cdot)\in\mathcal{U}_\mathbb{G}[s,T]$ and $\mathcal{D}J(s,x_0;u(\cdot))(t)=0, \forall t\in [s,T]$
\end{theorem}

\begin{proof}
	(i) For any $u(\cdot)$, $\nu(\cdot)\in\mathcal{U}_\mathbb{G}[s,T]$, let $x(\cdot)$ and $x^\lambda(\cdot)$ be the solution of system (\ref{system}) with control being $u(\cdot)$ and $u(\cdot)+\lambda\nu(\cdot)$, respectively. Then we know $\widetilde{x}\triangleq(x^\lambda-x)/\lambda$ satisfies
	\begin{equation}
		\begin{cases}
			\begin{split}
				d\widetilde{x}(t)= \,\,&\big[A\widetilde{x}(t)+B\nu(t)\big]dt+\big[C_1\widetilde{x}(t)+D_1\nu(t)\big]dW_1(t)\\
				&+\big[C_2\widetilde{x}(t)+D_2\nu(t)\big]dW_2(t),\quad t\in[s,T],\\
			\end{split}\\
			\widetilde{x}(s)=0.
		\end{cases}
	\end{equation}
	
	Applying It\^o's formula to $d(\widetilde{x}(t)^\top Y(t))$, we have
	\begin{equation}\label{eq20}
		\begin{split}
			&\mathbb{E}\big[\widetilde{x}^\top(T)(Gx(T)+g)\big]\\
			=&\mathbb{E}\int_{s}^{T}\Big\{[A\widetilde{x}+B\nu]^\top Y-\widetilde{x}^\top\big[A^\top Y+C_1^\top Z_1+C_2^\top Z_2+Qx+S^\top u+q\big]\\
			&+[C_1\widetilde{x}+D_1\nu]^\top Z_1+[C_2\widetilde{x}+D_2\nu]^\top Z_2\Big\}dt\\
			=&\mathbb{E}\int_{s}^{T}\Big\{\nu^\top B^\top Y+\widetilde{x}^\top\big[-Qx-S^\top u-q\big]+\nu^\top D_1^\top Z_1+\nu^\top D_2^\top Z_2\Big\}dt.
		\end{split}
	\end{equation}

	Based on (\ref{eq20}), it follows from cost functional (\ref{cost}) that 
	\begin{equation*}
		\begin{split}
			&J(s,x_0;u(\cdot)+\lambda \nu(\cdot))-J(s,x_0;u(\cdot))\\
			=\ &\lambda^2 J^0(s,0;\nu(\cdot))
			+2\lambda\mathbb{E}\big[\widetilde{x}^\top(T)(Gx(T)+g)\big]\\
			&+2\lambda\mathbb{E}\int_{s}^{T}\Big\{u^\top S\widetilde{x}+u^\top R \nu+\widetilde{x}^\top q+\nu^\top\rho+\widetilde{x}^\top Qx+\nu^\top Sx\Big\} dt\\
			=\ &\lambda^2 J^0(s,0;\nu(\cdot))
			+2\lambda\mathbb{E}\int_{s}^{T}\Big\{\nu^\top\big[B^\top Y+D_1^\top Z_1+D_2^\top Z_2+Ru+Sx+\rho\big]\Big\} dt.\\
		\end{split}
	\end{equation*}

(ii) It follows the definition of open-loop optimal control that $u(\cdot)$ is an open-loop optimal control if and only if
\begin{equation}
	\begin{split}
		J(s,x_0;u(\cdot)+\lambda \nu(\cdot))-J(s,x_0;u(\cdot))\geq 0, \quad\forall \nu(\cdot)\in\mathcal{U}_\mathbb{G}[s,T].
	\end{split}
\end{equation}
Thus, the statements in (ii) are direct results of (i). This completes the proof.
\end{proof}


To further investigate Problem (P), we introdce the following assumptions:

\begin{assum}\label{assum > gamma u}
	There exists a real number $\gamma>0$ such that 
	\begin{equation}\label{eq23}
		J^0(0,0;u(\cdot))\geq\gamma\mathbb{E}\int_{0}^{T}|u(t)|^2dt, \quad \forall u(\cdot)\in\mathcal{U}_\mathbb{G}[0,T].
	\end{equation} 
\end{assum}

\begin{assum}\label{ass_cost>0}
	$J^0(0,0;u(\cdot))\geq 0, \quad\forall u(\cdot)\in\mathcal{U}_\mathbb{G}[s,T].$
\end{assum}
Assumption \ref{assum > gamma u} is a uniformly convexity assumption for cost functional (\ref{cost}). By virtue of Theorem \ref{theorem_openLOOP_FBSDE}, Assumption \ref{ass_cost>0} is a necessary condition for the open-loop solvability of Problem (P). Similar assumptions also appears in some works studying indefinite LQSOC problems with full information, say, e.g., \cite{SunLiYong2016_SLQ,SunYong2014_Game,ZhangLiXiong2021,WenLiXiong2021}.

\begin{proposition}\label{proposition_1}
	Suppose that Assumptions \ref{assum_system}, \ref{assum_cost} and \ref{assum > gamma u} hold, then Problem (P) is uniquely open-loop solvable. Moreover, there exists a constant $\theta>0$ such that 
	\begin{equation}\label{eq22}
		V^0(s,x_0)\geq \theta|x_0|^2,\quad\forall(s,x_0)\in[0,T)\times\mathbb{R}^n.
	\end{equation}
\end{proposition}

\begin{proof}
For any $s\in[0,T)$ and $u(\cdot)\in\mathcal{U}_\mathbb{G}[s,T]$, we define
\begin{equation}
		\big[0\mathbb{I}_{[0,s]}\oplus u(\cdot)\big](t)=\begin{cases}
			0,\quad\quad \,t\in[0,s],\\
			u(t), \quad t\in[s,T].\\
		\end{cases}
\end{equation}
Thus, Assumption \ref{assum > gamma u} yields
\begin{equation}\label{eq25}
	\begin{split}
		J^0(s,0;u(\cdot))=J^0(0,0;0\mathbb{I}_{[0,s]}\oplus u(\cdot))\geq \gamma\mathbb{E}\int_{0}^{T}\big|[0\mathbb{I}_{[0,s]}\oplus u(\cdot)](t)\big|^2dt=\gamma\mathbb{E}\int_{s}^{T}|u(t)|^2dt.
	\end{split}
\end{equation}
According to Theorem \ref{theorem_openLOOP_FBSDE} and (\ref{eq25}), given any $u(\cdot)\in\mathcal{U}_\mathbb{G}[s,T]$, we obtain
\begin{equation}\label{eq26}
	\begin{split}
		&J(s,x_0;u(\cdot))=J(s,x_0;0)+J^0(s,0;u(\cdot
		))+\mathbb{E}\int_{s}^{T}u(t)^\top \mathcal{D}J(s,x_0;0)(t)dt\\
		\geq\ &J(s,x_0;0)+J^0(s,0;u(\cdot
		))-\frac{\gamma}{2}\mathbb{E}\int_{s}^{T}|u(t)|^2dt-\frac{1}{2\gamma}\mathbb{E}\int_{s}^{T}|\mathcal{D}J(s,x_0;0)(t)|^2dt\\
		\geq\ &J(s,x_0;0)+\frac{\gamma}{2}\mathbb{E}\int_{s}^{T}|u(t)|^2dt-\frac{1}{2\gamma}\mathbb{E}\int_{s}^{T}|\mathcal{D}J(s,x_0;0)(t)|^2dt.\\
	\end{split}
\end{equation} 
As a result, based on a standard argument involving minizing sequence and locally weak compactness of Hilbert spaces, Problem (P) has a unique open-loop optimal control for any $(s,x_0)\in[0,T)\times\mathbb{R}^n$. 

Further, when $b(\cdot)$, $\sigma_1(\cdot)$, $\sigma_2(\cdot)$, $q(\cdot)$, $\rho(\cdot)$, $g=0$, it follows from (\ref{eq26}) that
\begin{equation}\label{eq27}
	V^0(s,x_0)\geq J^0(s,x_0;0)-\frac{1}{2\gamma}\mathbb{E}\int_{s}^{T}|\mathcal{D}J^0(s,x_0;0)(t)|^2dt.
\end{equation}
Since the functions on the right-hand side of (\ref{eq27}) are continuous in $t$ and quadratic in $x_0$, we obtain (\ref{eq22}). The proof is then completed. 
\end{proof}

Next, we investigate Riccati equation (\ref{Riccati_2}) under Assumption \ref{assum > gamma u}.

\begin{theorem}\label{theorem_Riccati_assum4}
	Let Assumptions \ref{assum_system}, \ref{assum_cost} and \ref{assum > gamma u} hold. Then Riccati equation (\ref{Riccati_2}) admits a unique solution $P^2(\cdot)\in C([0,T];\mathbb{S}^n)$ such that 
	\begin{equation}\label{eq288}
		R(t)+D_1(t)^\top P^1(t)D_1(t)+D_2(t)^\top P^2(t)D_2(t)\geq \gamma\mathbb{I}_m,\quad a.e. \,\,t\in[0,T],
	\end{equation}
for some $\gamma>0$. In consequence, Problem (P) is uniquely closed-loop solvable and the optimal closed-loop strategy is  
	\begin{equation*}
		\begin{split}
			\varTheta^*\triangleq&-(R+D_1^\top P^1D_1+D_2^\top P^2D_2)^{-1}(B^\top P^2+D_1^\top P^1C_1+D_2^\top P^2C_2+S),\\
			\varLambda^*\triangleq&-(R+D_1^\top P^1D_1+D_2^\top P^2D_2)^{-1}(B^\top\alpha+D_2^\top\beta+D_1^\top P^1\sigma_1+D_2^\top P^2\sigma_2+\rho),\\
		\end{split}
	\end{equation*}
	where $(\alpha(\cdot),\beta(\cdot))$ is the adapted solution of BSDE (\ref{BSDE}). Moreover, Problem (P) is uniquely open-loop solvable and the open-loop optimal control for initial pair $(s,x_0)$ is 
	\begin{equation*}
		\begin{split}
			u^*(t)=\varTheta^*(t)\widehat{x}^*(t)+\varLambda^*(t),
		\end{split}
	\end{equation*}
	where $\widehat{x}^*(\cdot)$ is the solution to (\ref{eq6}) with $(\varTheta^*(\cdot),\varLambda^*(\cdot))$. 
\end{theorem}

Before giving the proof of Theorem \ref{theorem_Riccati_assum4}, we prove the following results.

\begin{lemma}\label{lemma 4}
	Let Assumptions \ref{assum_system}, \ref{assum_cost} hold and $\varTheta(\cdot)\in L^2(0,T;\mathbb{R}^{m\times n})$. Let $\widetilde{P}^1(\cdot), \widetilde{P}^2(\cdot)\in C([0,T];\mathbb{S}^n)$ be the solutions of
	\begin{equation}\label{eq28}
		\begin{cases}
			\begin{split}
				\dot{\widetilde{P}^1}&+\widetilde{P}^1(A+B\varTheta)+(A+B\varTheta)^\top \widetilde{P}^1+(C_1+D_1\varTheta)^\top \widetilde{P}^1(C_1+D_1\varTheta)\\
				&+(C_2+D_2\varTheta)^\top \widetilde{P}^1(C_2+D_2\varTheta)+Q=0,
			\end{split}\\
			P(T)=G,\\
		\end{cases}
	\end{equation}
and
\begin{equation}\label{eq29}
	\begin{cases}
		\begin{split}
			\dot{\widetilde{P}^2}&+\widetilde{P}^2(A+B\varTheta)+(A+B\varTheta)^\top \widetilde{P}^2+(C_1+D_1\varTheta)^\top \widetilde{P}^1(C_1+D_1\varTheta)\\
			&+(C_2+D_2\varTheta)^\top \widetilde{P}^2(C_2+D_2\varTheta)
			+\varTheta^\top R\varTheta+S^\top\varTheta+\varTheta^\top S+Q=0,
		\end{split}\\
		P(T)=G.\\
	\end{cases}
\end{equation}
Then we have

(i) Given any $(s,x_0)\in[0,T)\times\mathbb{R}^n$ and $u(\cdot)\in\mathcal{U}_\mathbb{G}[s,T]$, the following equation holds
\begin{equation}\label{eq30}
	\begin{split}
		&J^0(s,x_0;\varTheta(\cdot)\widehat{x}(\cdot)+u(\cdot))\\
		=\ &x_0^\top \widetilde{P}^2 x_0+\mathbb{E}\int_{s}^{T}\Big\{u^\top\big[R+D_1^\top \widetilde{P}^1 D_1+D_2^\top \widetilde{P}^2D_2\big]u\\
		&+2u^\top \big[B^\top\widetilde{P}^2+D_1^\top\widetilde{P}^1C_1+D_2^\top\widetilde{P}^2C_2+S+(R+D_1^\top\widetilde{P}^1D_1+D_2^\top\widetilde{P}^2D_2)\varTheta\big]\widehat{x}\Big\}dt,\\
	\end{split}
\end{equation}
where $\widehat{x}(\cdot)$ is the solution to 
\begin{equation}\label{eq31}
	\begin{cases}
		\begin{split}
			d\widehat{x}(t)= \,\,&\big[(A+B\varTheta)\widehat{x}(t)+Bu(t)\big]dt+
			\big[(C_2+D_2\varTheta)\widehat{x}(t)+D_2u(t)\big]dW_2(t),\quad t\in[s,T],\\
		\end{split}\\
		\widehat{x}(s)=x_0.
	\end{cases}
\end{equation}

(ii) If, in addition, Assumption \ref{assum > gamma u} holds, then the solution $\widetilde{P}^2(\cdot)$ to (\ref{eq29}) satisfies 
\begin{equation*}
	R(t)+D_1(t)^\top\widetilde{P}^1(t)D_1(t)+D_2(t)^\top\widetilde{P}^2(t)D_2(t)\geq\gamma\mathbb{I}_m,\quad \widetilde{P}^2(t)\geq \theta\mathbb{I}_n,\quad\forall t\in[0,T],
\end{equation*}
where $\theta>0$ and $\gamma>0$ are constants appearing in Proposition \ref{proposition_1} and (\ref{eq23}).
\end{lemma} 

\begin{proof}
	(i) Given any $(s,x_0)\in[0,T)\times\mathbb{R}^n$ and $u(\cdot)\in\mathcal{U}_\mathbb{G}[s,T]$, let $x(\cdot)$  evolve according to
	\begin{equation*}
		\begin{cases}
			\begin{split}
				dx(t)= \,\,&\big[(A+B\varTheta)x(t)+Bu(t)\big]dt+\big[(C_1+D_1\varTheta)x(t)+D_1u(t)\big]dW_1(t)\\
				&+\big[(C_2+D_2\varTheta)x(t)+D_2u(t)\big]dW_2(t),\quad t\in[s,T],\\
			\end{split}\\
			x(s)=x_0.
		\end{cases}
	\end{equation*}

Applying It\^o's formula to $d(x(t)^\top \widetilde{P}^1(t)x(t))$ and $d(\widehat{x}(t)^\top \big(\widetilde{P}^2(t)-\widetilde{P}^1(t)\big)\widehat{x}(t))$, we derive
\begin{equation*}
	\begin{split}
		&J^0(s,x_0;\varTheta(\cdot)\widehat{x}(\cdot)+u(\cdot))\\
		=\ &\mathbb{E}\int_{s}^{T}\Big\{x^\top\big[\dot{\widetilde{P}^1}+\widetilde{P}^1(A+B\varTheta)+(A+B\varTheta)^\top \widetilde{P}^1+(C_1+D_1\varTheta)^\top \widetilde{P}^1(C_1+D_1\varTheta)\\
		&+(C_2+D_2\varTheta)^\top \widetilde{P}^1(C_2+D_2\varTheta)+Q\big]x+u^\top\big[R+D_1^\top \widetilde{P}^1 D_1+D_2^\top \widetilde{P}^2D_2\big]u\\
		&+2u^\top \big[B^\top\widetilde{P}^2+D_1^\top\widetilde{P}^1C_1+D_2^\top\widetilde{P}^2C_2+S+(R+D_1^\top\widetilde{P}^1D_1+D_2^\top\widetilde{P}^2D_2)\varTheta\big]\widehat{x}\\
		&+\widehat{x}^\top\big[\dot{\widetilde{P}^2}-\dot{\widetilde{P}^1}+(\widetilde{P}^2-\widetilde{P}^1)(A+B\varTheta)+(A+B\varTheta)^\top (\widetilde{P}^2-\widetilde{P}^1)\\
		&+(C_2+D_2\varTheta)^\top (\widetilde{P}^2-\widetilde{P}^1)(C_2+D_2\varTheta)+\varTheta^\top R\varTheta+S^\top\varTheta+\varTheta^\top S\big]\widehat{x}\Big\}dt+x_0^\top \widetilde{P}^2 x_0\\
		=\ &\mathbb{E}\int_{s}^{T}\Big\{u^\top\big[R+D_1^\top \widetilde{P}^1 D_1+D_2^\top \widetilde{P}^2D_2\big]u
		+2u^\top \big[B^\top\widetilde{P}^2+D_1^\top\widetilde{P}^1C_1+D_2^\top\widetilde{P}^2C_2+S\\
		&+(R+D_1^\top\widetilde{P}^1D_1+D_2^\top\widetilde{P}^2D_2)\varTheta\big]\widehat{x}\Big\}dt+x_0^\top \widetilde{P}^2(s) x_0.\\
	\end{split}
\end{equation*}

(ii) For any $u(\cdot)\in\mathcal{U}_\mathbb{G}[0,T]$, let $\widehat{x}^0(\cdot)$ be the solution of (\ref{eq31}) with $\widehat{x}^0(0)=0$. By virtue of (i) and (\ref{eq23}), we know
\begin{equation}
	\begin{split}
		&J^0(0,0;\varTheta(\cdot)\widehat{x}^0(\cdot)+u(\cdot))\\
		=\ &\mathbb{E}\int_{0}^{T}\Big\{u^\top\big[R+D_1^\top \widetilde{P}^1 D_1+D_2^\top \widetilde{P}^2D_2\big]u\\
		&+2u^\top \big[B^\top\widetilde{P}^2+D_1^\top\widetilde{P}^1C_1+D_2^\top\widetilde{P}^2C_2+S+(R+D_1^\top\widetilde{P}^1D_1+D_2^\top\widetilde{P}^2D_2)\varTheta\big]\widehat{x}^0\Big\}dt\\
		\geq\ &\gamma\mathbb{E}\int_{0}^{T}\big|\varTheta\widehat{x}^0+u\big|^2dt.
	\end{split}
\end{equation}
The above equation means that 
\begin{equation}\label{eq33}
	\begin{split}
		&\mathbb{E}\int_{0}^{T}\Big\{u^\top\big[R+D_1^\top \widetilde{P}^1 D_1+D_2^\top \widetilde{P}^2D_2-\gamma\mathbb{I}_m\big]u\\
		&+2u^\top \big[B^\top\widetilde{P}^2+D_1^\top\widetilde{P}^1C_1+D_2^\top\widetilde{P}^2C_2+S+(R+D_1^\top\widetilde{P}^1D_1+D_2^\top\widetilde{P}^2D_2-\gamma\mathbb{I}_m)\varTheta\big]\widehat{x}^0\Big\}dt\\
		\geq\ &\gamma\mathbb{E}\int_{0}^{T}\big|\varTheta\widehat{x}^0\big|^2dt\geq 0,\quad \forall u(\cdot)\in\mathcal{U}_\mathbb{G}[0,T].
	\end{split}
\end{equation}

Further, for any fixed $u_0\in\mathbb{R}^m$, we adopt $u(t)=u_0\textbf{1}_{[s,s+h]}(t)$ with $0\leq s<s+h\leq T$. Hence 
\begin{equation*}
	\begin{cases}
		d\big(\mathbb{E}[\widehat{x}^0(t)]\big)=\Big\{\big[A+B\varTheta\big]\mathbb{E}[\widehat{x}^0(t)]+Bu_0\textbf{1}_{[s,s+h]}(t)\Big\}dt,\quad t\in[0,T],\\
		\mathbb{E}[\widehat{x}^0(0)]=0.\\
	\end{cases}
\end{equation*}
This implies
\begin{equation}\label{eq35}
	\mathbb{E}[\widehat{x}^0(t)]=\begin{cases}
		0,\quad\quad \quad \quad \quad \quad \quad \quad \quad \quad \quad \quad\,\, t\in[0,s],\\
		\varPi(t)\int_{s}^{t\land (s+h)}\varPi(r)B(r)u_0dr,\quad t\in[s,T],\\
	\end{cases}
\end{equation}
where $\varPi(\cdot)$ is the solution to
\begin{equation*}
	\begin{cases}
		\dot{\varPi}(t)=\big[A(t)+B(t)\varTheta(t)\big]\varPi(t)\\
		\varPi(0)=\mathbb{I}_n.\\
	\end{cases}
\end{equation*}

Combining (\ref{eq33}) with (\ref{eq35}), we derive
\begin{equation*}
	\begin{split}
		&\int_{s}^{s+h}\Big\{2u_0^\top \big[B^\top\widetilde{P}^2+D_1^\top\widetilde{P}^1C_1+D_2^\top\widetilde{P}^2C_2+S+(R+D_1^\top\widetilde{P}^1D_1+D_2^\top\widetilde{P}^2D_2-\gamma\mathbb{I}_m)\varTheta\big]\varPi(t)\int_{s}^{t}\varPi(r)B(r)u_0dr\\
		&+u_0^\top\big[R+D_1^\top \widetilde{P}^1 D_1+D_2^\top \widetilde{P}^2D_2-\gamma\mathbb{I}_m\big]u_0\Big\}dt\geq 0.
	\end{split}
\end{equation*}
Dividing this equation by $h$ and letting $h\rightarrow 0$, we get 
\begin{equation*}
	u_0^\top\big[R(t)+D_1(t)^\top \widetilde{P}^1(t) D_1(t)+D_2(t)^\top \widetilde{P}^2(t)D_2(t)-\gamma\mathbb{I}_m\big]u_0\geq0,\quad a.e. \,t\in[0,T], \forall u_0\in\mathbb{R}^n,
\end{equation*}
which yields the first inequality of (ii).

Moreover, (\ref{eq30}) and Proposition \ref{proposition_1} show that
\begin{equation*}
	\begin{split}
		\theta|x_0|^2\leq\ & V^0(s,x_0)\leq J^0(s,x_0;\varTheta(\cdot)\widehat{x}(\cdot)+u(\cdot))\\
		=\ &x_0^\top \widetilde{P}^2 x_0+\mathbb{E}\int_{s}^{T}\Big\{u^\top\big[R+D_1^\top \widetilde{P}^1 D_1+D_2^\top \widetilde{P}^2D_2\big]u\\
		&+2u^\top \big[B^\top\widetilde{P}^2+D_1^\top\widetilde{P}^1C_1+D_2^\top\widetilde{P}^2C_2+S+(R+D_1^\top\widetilde{P}^1D_1+D_2^\top\widetilde{P}^2D_2)\varTheta\big]\widehat{x}\Big\}dt.\\
	\end{split}
\end{equation*}
By taking $u(\cdot)=0$ in the above equation, we have $x_0^\top \widetilde{P}^2(s) x_0\geq \theta|x_0|^2$, $\forall(s,x_0)\in[0,T]\times\mathbb{R}^n$, which yields the second inequality of (ii). The proof is completed.
\end{proof}

Now we are ready to prove Theorem \ref{theorem_Riccati_assum4}.

\textit{Proof of Theorem \ref{theorem_Riccati_assum4}}. By virtue of Assumptions \ref{assum_system}, \ref{assum_cost} and  $\varTheta(\cdot)\in L^2(0,T;\mathbb{R}^{m\times n})$, it is clear that equation (\ref{eq28}) admits a unique solution $\widetilde{P}^1(\cdot)\in C(0,T;\mathbb{S}^n)$. By defining \begin{equation*}
	\begin{split}
		\mathscr{R}(\cdot)&\triangleq R(\cdot)+D_1(\cdot)^\top \widetilde{P}^1(\cdot) D_1(\cdot),\\
		\mathscr{Q}(\cdot)&\triangleq Q(\cdot)+\big(C_1(\cdot)+D_1(\cdot)\varTheta(\cdot)\big)^\top \widetilde{P}^1(\cdot)\big(C_1(\cdot)+D_1(\cdot)\varTheta(\cdot)\big)-\varTheta(\cdot)^\top D_1(\cdot)^\top \widetilde{P}^1(\cdot) D_1(\cdot) \varTheta(\cdot),\\
	\end{split}
\end{equation*}
then (\ref{eq29}) becomes
\begin{equation}\label{eq36}
	\begin{cases}
		\begin{split}
			\dot{\widetilde{P}^2}&+\widetilde{P}^2(A+B\varTheta)+(A+B\varTheta)^\top \widetilde{P}^2+(C_2+D_2\varTheta)^\top \widetilde{P}^2(C_2+D_2\varTheta)\\
			&+\varTheta^\top \mathscr{R}\varTheta+S^\top\varTheta+\varTheta^\top S+\mathscr{Q}=0,
		\end{split}\\
		P(T)=G.\\
	\end{cases}
\end{equation}
In this case, Lemma \ref{lemma 4} implies
\begin{equation}\label{eq37}
	\mathscr{R}(t)+D_2(t)^\top\widetilde{P}^2(t)D_2(t)\geq\gamma\mathbb{I}_m,\quad \widetilde{P}^2(t)\geq \theta\mathbb{I}_n,\quad\forall t\in[0,T].
\end{equation}
Hence,  (\ref{eq29}) can be transformed into Lyapunov equation (\ref{eq36}) with property (\ref{eq37}), which is the same as the one arising in LQSOC problems with full information in Sun et al. \cite{SunLiYong2016_SLQ}. Then inequality (\ref{eq288})  follows immediately from the proof of the first half of \cite[Theorem 4.5]{SunLiYong2016_SLQ}. Furthermore, with the help of (\ref{eq288}), the rest of this theorem can be proved by following the proof of Theorem \ref{optimal proof}. The proof is thus completed.$\hfill\qedsymbol$

\vspace{0.3cm}

Finally, we present a characterization of the open-loop solvability by virtue of a perturbation approach. To do this, we consider a perturbed version of Problem (P). Given any $\epsilon>0$, consider the LQSOC problem of minimizing a perturbed cost functional
\begin{equation}\label{cost_epsilon}
	\begin{split}
		J_\epsilon(s,x_0;u(\cdot))&\triangleq\ J(s,x_0;u(\cdot))+\epsilon\mathbb{E}\int_{s}^{T}|u(t)|^2dt\\
		&=\mathbb{E}\int_{s }^{T}\left[\left<\begin{pmatrix}
			Q(t) & S(t)^\top\\
			S(t) & R(t)+\epsilon\mathbb{I}_m\\
		\end{pmatrix}\begin{pmatrix}
			x(t)\\
			u(t)\\
		\end{pmatrix},\begin{pmatrix}
			x(t)\\
			u(t)\\
		\end{pmatrix}\right>+2\left<\begin{pmatrix}
			q(t)\\
			\rho(t)\\
		\end{pmatrix},\begin{pmatrix}
			x(t)\\
			u(t)\\
		\end{pmatrix}\right>\right]dt\\
		&+\mathbb{E}\left[\left<G x(T),x(T)\right> +2\left<g ,x(T)\right>\right],
	\end{split}
\end{equation}
with the system being (\ref{system}) and the admissible control set being $\mathcal{U}_\mathbb{G}[s,T]$. We denote this perturbed problem by Problem(P)$_\epsilon$ and the corresponding value function by $V_\epsilon(\cdot,\cdot)$. Similarly, when the nonhomogeneous terms of system  (\ref{system}) and cost functional (\ref{cost_epsilon}) disappear, we denote the cost functional  by $J_\epsilon^0(s,x_0;u(\cdot))$. It is worth stressing that 
\begin{equation*}
	J_\epsilon^0(s,x_0;u(\cdot))=J(s,x_0;u(\cdot))+\epsilon\mathbb{E}\int_{s}^{T}|u(t)|^2dt,
\end{equation*} 
which, under Assumption \ref{ass_cost>0}, satisfies
\begin{equation*}
	J_\epsilon^0(0,0;u(\cdot))\geq\epsilon\mathbb{E}\int_{0}^{T}|u(t)|^2dt,\quad\forall u(\cdot)\in\mathcal{U}_\mathbb{G}[0,T].
\end{equation*}

According to Theorem \ref{theorem_Riccati_assum4}, we have the following lemma.
\begin{lemma}\label{lemma 5}
	Let Assumptions \ref{assum_system}, \ref{assum_cost}, \ref{ass_cost>0} hold. For any $(s,x_0)\in[0,T)\times\mathbb{R}^n$, the optimal control of Problem (P)$_\epsilon$ is 
	\begin{equation*}
		\begin{split}
			u_\epsilon(t)=\varTheta_\epsilon(t)\widehat{x}_\epsilon(t)+\varLambda_\epsilon(t),
		\end{split}
	\end{equation*}
	where 
	\begin{equation}\label{eq_varTheta_varLambda}
		\begin{split}
			\varTheta_\epsilon\triangleq&-(R+\epsilon\mathbb{I}_m+D_1^\top P^1 D_1+D_2^\top P^2_\epsilon D_2)^{-1}(B^\top P^2_\epsilon+D_1^\top P^1 C_1+D_2^\top P^2_\epsilon C_2+S),\\
			\varLambda_\epsilon\triangleq&-(R+\epsilon\mathbb{I}_m+D_1^\top P^1 D_1+D_2^\top P^2_\epsilon D_2)^{-1}(B^\top\alpha_\epsilon+D_2^\top\beta_\epsilon+D_1^\top P^1\sigma_1+D_2^\top P^2_\epsilon\sigma_2+\rho),\\
		\end{split}
	\end{equation} 
$P^1(\cdot)$ is the solution to (\ref{Riccati_1}), $P^2_\epsilon(\cdot)$ is the solution to 
\begin{equation}\label{Riccati_2_epsilon}
	\begin{cases}
		\begin{split}
			\dot{P}^2_\epsilon&+A^\top P^2_\epsilon +P^2_\epsilon A+C_1^\top P^1 C_1+C_2^\top P^2_\epsilon C_2\\
			&-(B^\top P^2_\epsilon+D_1^\top P^1 C_1+D_2^\top P^2_\epsilon C_2+S)^\top(R+\epsilon\mathbb{I}_m+D_1^\top P^1D_1+D_2^\top P^2_\epsilon D_2)^{-1}\\
			&\times(B^\top P^2_\epsilon+D_1^\top P^1C_1+D_2^\top P^2_\epsilon C_2+S)+Q=0,\\	
		\end{split}\\
		P^2_\epsilon(T)=G,\\
	\end{cases}
\end{equation}
$(\alpha_\epsilon(\cdot),\beta_\epsilon(\cdot))$ is the adapted solution of 
\begin{equation}\label{BSDE_epsilon}
	\begin{cases}
		\begin{split}
			d\alpha_\epsilon(t)=&-\Big[\big(A+B\varTheta_\epsilon\big)^\top\alpha_\epsilon(t)
			+\big(C_2+D_2\varTheta_\epsilon\big)^\top\beta_\epsilon(t)
			+\big(C_1+D_1\varTheta_\epsilon\big)^\top P^1\sigma_1\\
			&+\big(C_2+D_2\varTheta_\epsilon\big)^\top P^2_\epsilon\sigma_2
			+\varTheta_\epsilon^\top\rho+P^2_\epsilon b+q\Big]dt+\beta_\epsilon(t) dW_2(t),\\
		\end{split}\\
		\alpha_\epsilon(T)=g,\\
	\end{cases}
\end{equation}
 and $\widehat{x}_\epsilon(\cdot)$ evolves according to
	\begin{equation}\label{eq400}
		\begin{cases}
			\begin{split}
				d\widehat{x}_\epsilon(t)= \big[A\widehat{x}_\epsilon(t)+Bu_\epsilon(t)+b\big]dt
				+\big[C_2\widehat{x}_\epsilon(t)+D_2u_\epsilon(t)+\sigma_2\big]dW_2(t),\quad t\in[s,T],
			\end{split}\\
			\widehat{x}_\epsilon(s)=x_0.\\
		\end{cases}
	\end{equation}
Moreover, the solution $P^2_\epsilon(\cdot)\in C(0,T;\mathbb{S}^n)$ satsifies 
	\begin{equation*}
	R(t)+\epsilon\mathbb{I}_m+D_1(t)P^1(t)D_1(t)+D_2(t)P^2_\epsilon(t)D_2(t)> 0,\quad a.e. \,\,t\in[0,T].
\end{equation*}
\end{lemma}

Based on the above lemma, we now study the open-loop solvability of Problem (P) by means of the perturbed control $u_\epsilon(\cdot)$ constructed in Lemma \ref{lemma 5}.

\begin{theorem}\label{theorem_4}
	Assume that Assumptions \ref{assum_system}, \ref{assum_cost}, \ref{ass_cost>0} hold. Given $(s,x_0)\in[0,T)\times\mathbb{R}^n$, consider  $\{u_\epsilon(\cdot)\}_{\epsilon>0}$ constructed by Lemma \ref{lemma 5}, then the following three statements are equivalent:
	
(i) Problem (P) is open-loop solvable at $(s,x_0)$.

(ii) The collection $\{u_\epsilon(\cdot)\}_{\epsilon>0}$ is bounded in $L^2_\mathbb{G}(s,T;\mathbb{R}^m)$. That is 
\begin{equation*}
	\sup\limits_{\epsilon>0}\mathbb{E}\int_{s}^{T}|u_\epsilon(t)|^2dt<\infty.
\end{equation*}

(iii) The collection $\{u_\epsilon(\cdot)\}_{\epsilon>0}$ converges strongly in $L^2_\mathbb{G}(s,T;\mathbb{R}^m)$ as $\epsilon\rightarrow 0$.
\end{theorem}

The next lemma reveals the relationship between the value functions of Problem (P)$_\epsilon$ and Problem (P), which plays a core role in proving Theorem \ref{theorem_4}.
\begin{lemma}
	Let Assumptions \ref{assum_system}, \ref{assum_cost} hold. Given $(s,x_0)\in[0,T)\times\mathbb{R}^n$, we have
	\begin{equation}\label{eq40}
		\lim\limits_{\epsilon\downarrow 0} V_\epsilon(s,x_0)=V(s,x_0).
	\end{equation}
\end{lemma}

\begin{proof}
	Given any fixed $(s,x_0)\in[0,T)\times\mathbb{R}^n$, any $\epsilon>0$ and any $u(\cdot)\in\mathcal{U}_\mathbb{G}[s,T]$, we know
\begin{equation*}
	\begin{split}
		J_\epsilon(s,x_0;u(\cdot))\triangleq\ J(s,x_0;u(\cdot))+\epsilon\mathbb{E}\int_{s}^{T}|u(t)|^2dt\geq J(s,x_0;u(\cdot))\geq V(s,x_0).
	\end{split}
\end{equation*}
This implies
\begin{equation}\label{eq41}
	 V_\epsilon(s,x_0)\geq V(s,x_0).
\end{equation}

Moreover, if $V(s,x_0)$ is finite, then for any $\delta>0$, there exists a $u^\delta(\cdot)\in\mathcal{U}_\mathbb{G}[s,T]$ such that $J(s,x_0;u^\delta(\cdot))\leq V(s,x_0)+\delta$. Hence,
\begin{equation*}
	V_\epsilon(s,x_0)\leq J(s,x_0;u^\delta(\cdot))+\epsilon\mathbb{E}\int_{s}^{T}|u^\delta(t)|^2dt\leq V(s,x_0)+\delta+\epsilon\mathbb{E}\int_{s}^{T}|u^\delta(t)|^2dt.
\end{equation*}
Taking $\epsilon\rightarrow 0$, we know
\begin{equation}\label{eq42}
	\lim\limits_{\epsilon\downarrow 0}V_\epsilon(s,x_0)\leq V(s,x_0)+\delta.
\end{equation}
Since $\delta>0$ is independent of $\epsilon$ and is arbitrary, (\ref{eq40}) follows from (\ref{eq41}) and (\ref{eq42}). Similar analysis also works if $V(s,x_0)=-\infty$. The proof is then complete. 
\end{proof}

\textit{Proof of Theorem \ref{theorem_4}}. First, we prove (i) $\Rightarrow$ (ii). Given any $(s,x_0)\in[0,T)\times\mathbb{R}^n$ and any $\epsilon>0$, if Problem (P) admits an open-loop optimal control $u^*(\cdot)$, then
\begin{equation}\label{eq43}
	\begin{split}
			V_\epsilon(s,x_0)\leq J_\epsilon(s,x_0;u^*(\cdot))=J(s,x_0;u^*(\cdot))+\epsilon\mathbb{E}\int_{s}^{T}|u^*(t)|^2dt= V(s,x_0)+\epsilon\mathbb{E}\int_{s}^{T}|u^*(t)|^2dt.
	\end{split}
\end{equation}
On the other hand, we know
\begin{equation}\label{eq44}
	\begin{split}
		V_\epsilon(s,x_0)= J_\epsilon(s,x_0;u_\epsilon(\cdot))=J(s,x_0;u_\epsilon(\cdot))+\epsilon\mathbb{E}\int_{s}^{T}|u_\epsilon(t)|^2dt\geq V(s,x_0)+\epsilon\mathbb{E}\int_{s}^{T}|u_\epsilon(t)|^2dt.
	\end{split}
\end{equation}
Combining (\ref{eq43}) with (\ref{eq44}), we derive
\begin{equation}\label{eq45}
	\mathbb{E}\int_{s}^{T}|u_\epsilon(t)|^2dt\leq\frac{V_\epsilon(s,x_0)-V(s,x_0)}{\epsilon}\leq\mathbb{E}\int_{s}^{T}|u^*(t)|^2dt,
\end{equation}
which yields the boundedness of $\{u_\epsilon(\cdot)\}_{\epsilon>0}$ in $L^2_\mathbb{G}(s,T;\mathbb{R}^m)$.

Next, we prove (ii) $\Rightarrow$ (i). If $\{u_\epsilon(\cdot)\}_{\epsilon>0}\subseteq L^2_\mathbb{G}(s,T;\mathbb{R}^m)$ is bounded, then there exists a sequence $\{\epsilon_i\}_{i=1}^\infty$ with $\lim_{i\rightarrow\infty}\epsilon_i=0$ such that $\{u_{\epsilon_i}(\cdot)\}_{i=1}^\infty$ converges weakly to a control $v^*(\cdot)\in L^2_\mathbb{G}(s,T;\mathbb{R}^m)$. Since the mapping $u(\cdot)\rightarrow J(s,x_0;u(\cdot))$ is continuous and convex, it is sequentially weakly lower semicontinuous. According to the boundedness of $\{u_{\epsilon_i}(\cdot)\}_{i=1}^\infty$ and (\ref{eq40}), we derive 
\begin{equation*}
	\begin{split}
		J(s,x_0;v^*(\cdot))& \leq \liminf\limits_{i\rightarrow\infty}J(s,x_0;u_{\epsilon_i}(\cdot))\\
		&=\liminf\limits_{i\rightarrow\infty}\left[V_\epsilon(s,x_0)-\epsilon_i\mathbb{E}\int_{s}^{T}|u_{\epsilon_i}(t)|^2dt\right]=V(s,x_0).
	\end{split}
\end{equation*}
This implies that $v^*(\cdot)$ is an open-loop optimal control of Problem (P).

Since the implication (iii) $\Rightarrow$ (ii) is trivially true, we finally prove  (ii) $\Rightarrow$ (iii). We prove it by the following two steps.

\textbf{Step 1.} Given any $(s,x_0)\in[0,T)\times\mathbb{R}^n$, the collection $\{u_\epsilon(\cdot)\}_{\epsilon>0}$ converges weakly to an open-loop optimal control of Problem (P) as $\epsilon\rightarrow 0$.

To prove it, we only need to show that any weakly convergent subsequence of $\{u_\epsilon(\cdot)\}_{\epsilon>0}$ has the same weak limit and this limit is an open-loop optimal control of Problem (P). We denote by $u^*_{j}(\cdot), j=1,2$, the weak limits of two different weakly convergent subsequences $\{u_{j,\epsilon_i}(\cdot)\}_{i=1}^\infty, j=1,2$, of $\{u_\epsilon(\cdot)\}_{\epsilon>0}$. Following similar procedures of (ii) $\Rightarrow$ (i), it is evident that $u^*_1(\cdot)$ and $u^*_2(\cdot)$ are optimal controls of Problem (P) for $(s,x_0)$.
Hence, the convexity of $u(\cdot)\rightarrow J(s,x_0,u(\cdot))$ shows
\begin{equation*}
	J\left(s,x_0;\frac{u^*_1(\cdot)+u^*_2(\cdot)}{2}\right)\leq\frac{1}{2}J\left(s,x_0;u^*_1(\cdot)\right)+\frac{1}{2}J\left(s,x_0;u^*_2(\cdot)\right)=V(s,x_0).
\end{equation*}
This implies that $(u^*_1(\cdot)+u^*_2(\cdot))/2$ is also an optimal control.

Then, similar to (\ref{eq45}), we obtain
\begin{equation*}
	\mathbb{E}\int_{s}^{T}|u_{j,\epsilon_i}(t)|^2dt\leq\mathbb{E}\int_{s}^{T}\left|\frac{u^*_1(\cdot)+u^*_2(\cdot)}{2}\right|^2dt,\quad j=1,2.
\end{equation*}
Taking inferior limits yields 
\begin{equation*}
	\mathbb{E}\int_{s}^{T}|u_{j}^*(t)|^2dt\leq\mathbb{E}\int_{s}^{T}\left|\frac{u^*_1(\cdot)+u^*_2(\cdot)}{2}\right|^2dt,\quad j=1,2.
\end{equation*}
By virtue of the above equation, we derive
\begin{equation*}
	2\left(\mathbb{E}\int_{s}^{T}|u_{1}^*(t)|^2dt+\mathbb{E}\int_{s}^{T}|u_{2}^*(t)|^2dt\right)\leq\mathbb{E}\int_{s}^{T}\left|u^*_1(\cdot)+u^*_2(\cdot)\right|^2dt,
\end{equation*}
which, by simple calculations, means
\begin{equation*}
	\mathbb{E}\int_{s}^{T}\left|u^*_1(\cdot)-u^*_2(\cdot)\right|^2dt\leq 0.
\end{equation*}
Therefore, we know $u^*_1(\cdot)=u^*_2(\cdot)$ and the statements in Step 1 holds.

\textbf{Step 2.} The collection $\{u_\epsilon(\cdot)\}_{\epsilon>0}$ converges strongly as $\epsilon\rightarrow 0$.

It follows from Step 1 that the collection $\{u_\epsilon(\cdot)\}_{\epsilon>0}$ converges weakly to an open-loop optimal control $u^*(\cdot)$ of Problem (P) for $(s,x_0)$ as $\epsilon\rightarrow 0$. According to (\ref{eq45}), we know
\begin{equation}\label{eq46}
	\mathbb{E}\int_{s}^{T}|u_\epsilon(t)|^2dt\leq\mathbb{E}\int_{s}^{T}|u^*(t)|^2dt,\quad\forall\epsilon>0.
\end{equation}
Moreover, since $\{u_\epsilon(\cdot)\}_{\epsilon>0}$ converges weakly to $u^*(\cdot)$, we obtain
\begin{equation}\label{eq47}
	\mathbb{E}\int_{s}^{T}|u^*(t)|^2dt\leq\liminf\limits_{\epsilon\rightarrow 0}\mathbb{E}\int_{s}^{T}|u_\epsilon(t)|^2dt,\quad\forall\epsilon>0.
\end{equation}
Combining (\ref{eq46}) with (\ref{eq47}), it is clear that  $\mathbb{E}\int_{s}^{T}|u_\epsilon(t)|^2dt$ has a limit $\mathbb{E}\int_{s}^{T}|u^*(t)|^2dt$. Hence, we derive
\begin{equation*}
	\begin{split}
		&\lim\limits_{\epsilon\rightarrow0}\mathbb{E}\int_{s}^{T}|u_\epsilon(t)-u^*(t)|^2dt\\
		=\ &\lim\limits_{\epsilon\rightarrow0}\left[\mathbb{E}\int_{s}^{T}|u_\epsilon(t)|^2dt+\mathbb{E}\int_{s}^{T}|u^*(t)|^2dt-2\mathbb{E}\int_{s}^{T}u^*(t)^\top u_\epsilon(t) dt\right]=0.
	\end{split}
\end{equation*}
This shows  $\{u_\epsilon(\cdot)\}_{\epsilon>0}$ converges strongly to $u^*(\cdot)$. The proof is thus completed.$\hfill\qedsymbol$

\section{Weak closed-loop solvabilities}\label{sec 4}
This section shows the equivalence between weak closed-loop solvability and open-loop solvability of Problem (P). We will prove that $(\varTheta_\epsilon(\cdot),\varLambda_\epsilon(\cdot))$ established in Lemma \ref{lemma 5} converges locally in $[0,T)$, and its limit pair $(\varTheta_\epsilon^*(\cdot),\varLambda_\epsilon^*(\cdot))$ is a weak closed-loop optimal strategy.

Let us begin with the next lemma, which is related to Problem (P)$^0$.
\begin{lemma}\label{lemma 7}
	Let Assumptions \ref{assum_system}, \ref{assum_cost}, \ref{ass_cost>0} hold. If Problem (P) is open-loop solvable, then so is Problem (P)$^0$.
\end{lemma}
\begin{proof}
	Given $(s,x_0)\in[0,T)\times\mathbb{R}^n$, when all nonhomogeneous terms in system (\ref{system}) and cost functional (\ref{cost_epsilon}) disappear, then BSDE (\ref{BSDE_epsilon}) admits an adapted solution $(\alpha_\epsilon(\cdot),\beta_\epsilon(\cdot))=(0,0)$ and thus $\varLambda_\epsilon(\cdot)=0$. In order to prove the open-loop solvability of Problem (P)$^0$,  Theorem \ref{theorem_4} implies that we can verify the boundedness of $\{u_\epsilon(\cdot)\}_{\epsilon>0}$ in $L^2_\mathbb{G}(s,T;\mathbb{R}^m)$ with $u_\epsilon(\cdot)=\varTheta_\epsilon(\cdot)\widehat{x}(\cdot)$ and $\widehat{x}_\epsilon(\cdot)$ being the solution to
\begin{equation}\label{eq50}
	\begin{cases}
		\begin{split}
			d\widehat{x}_\epsilon(t)= \big[(A+B\varTheta_\epsilon)\widehat{x}_\epsilon(t)\big]dt
			+\big[(C_2+D_2\varTheta_\epsilon)\widehat{x}_\epsilon(t)\big]dW_2(t),\quad t\in[s,T],
		\end{split}\\
		\widehat{x}_\epsilon(s)=x_0.\\
	\end{cases}
\end{equation}

To do this,  we define $\varLambda_\epsilon$ as that in (\ref{eq_varTheta_varLambda}) and let  $\widehat{x}_\epsilon^{s,x_0}(\cdot)$ and $\widehat{x}_\epsilon^{s,0}(\cdot)$ be the solutions of (\ref{eq400}) with initial pairs $(s,x_0)$ and $(s,0)$. Since Problem (P) is open-loop solvable at $(s,x_0)$ and $(s,0)$, it follows from Theorem \ref{theorem_4} that the following collections
\begin{equation*}
	u_\epsilon^{s,x_0}(\cdot)\triangleq\varTheta_\epsilon(\cdot)\widehat{x}_\epsilon^{s,x_0}(\cdot),\quad u_\epsilon^{s,0}(\cdot)\triangleq\varTheta_\epsilon(\cdot)\widehat{x}_\epsilon^{s,0}(\cdot),
\end{equation*}
are bounded in $L^2_\mathbb{G}(s,T;\mathbb{R}^n)$. Moreover, the definitions of $\widehat{x}_\epsilon^{s,x_0}(\cdot)$ and $\widehat{x}_\epsilon^{s,0}(\cdot)$ show  that $\widehat{x}_\epsilon^{s,x_0}(\cdot)-\widehat{x}_\epsilon^{s,0}(\cdot)$ satisfies (\ref{eq50}). By means of the uniqueness of solutions to SDEs, we obtain $\widehat{x}_\epsilon(\cdot)=\widehat{x}_\epsilon^{s,x_0}(\cdot)-\widehat{x}_\epsilon^{s,0}(\cdot)$ and thus $u_\epsilon(\cdot)=u_\epsilon^{s,x_0}(\cdot)-u_\epsilon^{s,0}(\cdot)$. Combining this with the fact that $\{u_\epsilon^{s,x_0}(\cdot)\}_{\epsilon>0}$ and $\{u_\epsilon^{s,0}(\cdot)\}_{\epsilon>0}$ are bounded in $L^2_\mathbb{G}(s,T;\mathbb{R}^n)$, we know $\{u_\epsilon(\cdot)\}_{\epsilon>0}$ also is bounded  in $L^2_\mathbb{G}(s,T;\mathbb{R}^n)$, which yields the open-loop solvability of Problem (P)$^0$.  
\end{proof}

Next, we show that the collection $\{\varTheta_\epsilon(\cdot)\}_{\epsilon>0}$ constructed in (\ref{eq_varTheta_varLambda}) is locally convergent in $[0,T)$.

\begin{proposition}\label{proposition 2}
	Assume that Assumptions \ref{assum_system}, \ref{assum_cost}, \ref{ass_cost>0} hold and Problem (P)$^0$ is open-loop solvable. Then the collection $\{\varTheta_\epsilon(\cdot)\}_{\epsilon>0}$ constructed in (\ref{eq_varTheta_varLambda}) converges in $L^2(0,T^{'};\mathbb{R}^{m\times n})$ for any $0\leq T^{'}<T$. That means, there is a locally square-integrable deterministic function $\varTheta_\epsilon^*:[0,T)\rightarrow\mathbb{R}^{m\times n}$ such that
	\begin{equation*}
		\lim\limits_{\epsilon\rightarrow 0}\int_{0}^{T^{'}}\Big|\varTheta_\epsilon(t)-\varTheta_\epsilon^*(\cdot)\Big|^2dt=0,\quad\forall\, 0<T^{'}<T.
	\end{equation*}
\end{proposition}

\begin{proof}
	Given any $0\leq T^{'}<T$, we only need to prove that the collection $\{\varTheta_\epsilon(\cdot)\}_{\epsilon>0}$ is Cauchy in $L^2(0,T^{'};\mathbb{R}^{m\times n})$. To do this,  given any fixed $s\in[0,T)$, let $\Psi_\epsilon(\cdot)\in L^2_{\mathbb{G}}(\Omega;C([s,T];\mathbb{R}^{n\times n}))$ be the solution to
	\begin{equation}
		\begin{cases}
			d\Psi_\epsilon(t)=(A+B\varTheta_\epsilon)\Psi_\epsilon(t)dt+(C_2+D_2\varTheta_\epsilon)\Psi_\epsilon(t)dW_2(t),\quad t\in[s,T],\\
			\Psi_\epsilon(s)=\mathbb{I}_n.
		\end{cases}
	\end{equation}
Then, given any initial state $x_0\in\mathbb{R}^n$, the solution to (\ref{eq50}) is $\widehat{x}_\epsilon(\cdot)=\Psi_\epsilon(\cdot)x_0$.
According to the open-loop solvability of Problem (P)$^0$ and Theorem \ref{theorem_4}, the set 
\begin{equation}
u_\epsilon(\cdot)=\varTheta_\epsilon(\cdot)\widehat{x}_\epsilon(\cdot)=\varTheta_\epsilon(\cdot)\Psi_\epsilon(\cdot)x_0, \quad\epsilon>0,
\end{equation}
converges strongly in $L^2_\mathbb{G}(s,T;\mathbb{R}^m)$ as $\epsilon\rightarrow 0$. This means that $\{\varTheta_\epsilon(\cdot)\Psi_\epsilon\}_{\epsilon>0}$ is strongly convergent in $L^2_\mathbb{G}(s,T;\mathbb{R}^{m\times n})$ as $\epsilon\rightarrow 0$. Let $\mathbb{U}_\epsilon(\cdot)\triangleq\varTheta_\epsilon(\cdot)\Psi_\epsilon(\cdot)$ and $\mathbb{U}^*(\cdot)$ be the corresponding  strong limit of $\mathbb{U}_\epsilon(\cdot)$. Then $\mathbb{E}[\Psi_\epsilon(\cdot)]$ satisfies
\begin{equation*}
	\begin{cases}
		d\mathbb{E}[\Psi_\epsilon(t)]=\left\{A\mathbb{E}[\Psi_\epsilon(t)]+B\mathbb{E}[\mathbb{U}_\epsilon(t)]\right\}dt,\quad t\in[s,T],\\
		\mathbb{E}[\Psi_\epsilon(s)]=\mathbb{I}_n,\\
	\end{cases}
\end{equation*}
and the Jensen's inequality shows 
\begin{equation*}
	\int_{s}^{T}\left|E[\mathbb{U}_\epsilon(t)]-E[\mathbb{U}^*_\epsilon(t)]\right|^2dt \leq\mathbb{E}\int_{s}^{T}\left|\mathbb{U}_\epsilon(t)-\mathbb{U}^*_\epsilon(t)\right|^2dt\rightarrow0,\quad\epsilon\rightarrow0.
\end{equation*}
Moreover, it follows from the theory of ODE that  $\mathbb{E}[\Psi_\epsilon(\cdot)]$ converges uniformly to the solution to
\begin{equation*}
	\begin{cases}
		d\mathbb{E}[\Psi^*(t)]=\left\{A\mathbb{E}[\Psi^*(t)]+B\mathbb{E}[\mathbb{U}^*(t)]\right\}dt,\quad t\in[s,T],\\
		\mathbb{E}[\Psi^*(s)]=\mathbb{I}_n.\\
	\end{cases}
\end{equation*}
Hence, since $\mathbb{E}[\Psi^*(s)]=\mathbb{I}_n$, we can find a small enough $\varDelta_s>0$ such that, for small enough $\epsilon>0$,  $\mathbb{E}[\Psi_\epsilon(t)]$, $\forall t\in[s,s+\varDelta_s]$, is invertible and $|\mathbb{E}[\Psi_\epsilon(t)]|\geq \frac{1}{2}$, $\forall t\in[s,s+\varDelta_s]$. 

Next, we prove that $\{\varTheta_\epsilon(\cdot)\}_{\epsilon>0}$ is Cauchy in $L^2(s,s+\varDelta_s;\mathbb{R}^{m\times n})$. In fact, for two small enough $\epsilon_1>0$ and $\epsilon_2>0$, it follows  that
\begin{equation*}
	\begin{split}
		&\int_{s}^{s+\varDelta_s}\Big|\varTheta_{\epsilon_1}(t)-\varTheta_{\epsilon_2}(t)\Big|^2dt\\
		\leq\ &\int_{s}^{s+\varDelta_s}\Big|\mathbb{E}[\mathbb{U}_{\epsilon_1}(t)]\mathbb{E}[\Psi_{\epsilon_1}(t)]^{-1}-\mathbb{E}[\mathbb{U}_{\epsilon_2}(t)]\mathbb{E}[\Psi_{\epsilon_2}(t)]^{-1}\Big|^2dt\\
		\leq\ &2\int_{s}^{s+\varDelta_s}\Big|\mathbb{E}[\mathbb{U}_{\epsilon_1}(t)-\mathbb{U}_{\epsilon_2}(t)]\Big|^2\Big|\mathbb{E}[\Psi_{\epsilon_1}(t)]^{-1}\Big|^2dt\\
		&+2\int_{s}^{s+\varDelta_s}\Big|\mathbb{E}[\mathbb{U}_{\epsilon_2}(t)]\Big|^2\Big|\mathbb{E}[\Psi_{\epsilon_1}(t)]^{-1}-\mathbb{E}[\Psi_{\epsilon_2}(t)]^{-1}\Big|^2dt\\
		=\ &2\int_{s}^{s+\varDelta_s}\Big|\mathbb{E}[\mathbb{U}_{\epsilon_1}(t)-\mathbb{U}_{\epsilon_2}(t)]\Big|^2\Big|\mathbb{E}[\Psi_{\epsilon_1}(t)]^{-1}\Big|^2dt\\
		&+2\int_{s}^{s+\varDelta_s}\Big|\mathbb{E}[\mathbb{U}_{\epsilon_2}(t)]\Big|^2\Big|\mathbb{E}[\Psi_{\epsilon_1}(t)]^{-1}\Big|^2\Big|\mathbb{E}[\Psi_{\epsilon_2}(t)]-\mathbb{E}[\Psi_{\epsilon_1}(t)]\Big|^2\Big|\mathbb{E}[\Psi_{\epsilon_2}(t)]^{-1}\Big|^2dt\\
		\leq\ &8\int_{s}^{s+\varDelta_s}\Big|\mathbb{E}[\mathbb{U}_{\epsilon_1}(t)-\mathbb{U}_{\epsilon_2}(t)]\Big|^2dt\\
		&+32\left(\sup\limits_{s\leq t\leq s+\varDelta_s}\Big|^2\Big|\mathbb{E}[\Psi_{\epsilon_2}(t)]-\mathbb{E}[\Psi_{\epsilon_1}(t)]\Big|^2\right)\int_{s}^{s+\varDelta_s}\Big|\mathbb{E}[\mathbb{U}_{\epsilon_2}(t)]\Big|^2dt. 
	\end{split}
\end{equation*}
Because $\{\mathbb{U}_\epsilon(\cdot)\}_{\epsilon>0}$ is Cauchy in $L^2_\mathbb{G}(s,T;\mathbb{R}^{m\times n})$ and $\mathbb{E}[\Psi_\epsilon(\cdot)]$ converges uniformly on $[s,T]$, the last two terms of the above inequality tends to $0$ as $\epsilon_1$, $\epsilon_2\rightarrow0$. This shows that $\{\varTheta_\epsilon(\cdot)\}_{\epsilon>0}$ is Cauchy in $L^2(s,s+\varDelta_s;\mathbb{R}^{m\times n})$.

Finally, we prove that $\{\varTheta_\epsilon(\cdot)\}_{\epsilon>0}$ is Cauchy in $L^2(0,T^{'};\mathbb{R}^{m\times n})$ by means of a compactness argument.  Given any $s\in[0,T^{'}]$, the above analysis implies that there exists a small enough $\varDelta_s>0$ such that $\{\varTheta_\epsilon(\cdot)\}_{\epsilon>0}$ is Cauchy in $L^2(s,s+\varDelta_s;\mathbb{R}^{m\times n})$. By the fact that $[0,T^{'}]$ is compact, we can find finitely many $s\in[0,T^{'}]$, say, $s_1, s_2,\cdots,s_i$, such that $\{\varTheta_\epsilon(\cdot)\}_{\epsilon>0}$ is Cauchy in $L^2(s_j,s_j+\varDelta_{s_j};\mathbb{R}^{m\times n})$ and $[0,T^{'}]\subseteq \cup_{j=1}^i[s_j,s_j+\varDelta_{s_j}]$. Therefore,  we have 
\begin{equation*}
	\int_{0}^{T^{'}}\Big|\varTheta_{\epsilon_1}(t)-\varTheta_{\epsilon_2}(t)\Big|^2dt\leq\sum_{j=1}^{i}\int_{s_i}^{s_i+\varDelta_{s_i}}\Big|\varTheta_{\epsilon_1}(t)-\varTheta_{\epsilon_2}(t)\Big|^2dt\rightarrow 0,\quad \text{as}\,\, \epsilon_1,\epsilon_2\rightarrow0.
\end{equation*}
This completes the proof.
\end{proof}

Further, we prove that the collection $\{\varLambda_\epsilon(\cdot)\}_{\epsilon>0}$ constructed in (\ref{eq_varTheta_varLambda}) is also locally convergent in $[0,T)$.

\begin{proposition}\label{proposition 3}
	Assume that Assumptions \ref{assum_system}, \ref{assum_cost}, \ref{ass_cost>0} hold and Problem (P) is open-loop solvable. Then the collection $\{\varLambda_\epsilon(\cdot)\}_{\epsilon>0}$ constructed in (\ref{eq_varTheta_varLambda}) converges in $L^2_\mathbb{G}(0,T^{'};\mathbb{R}^{m})$ for any $0\leq T^{'}<T$. That means, there is a locally square-integrable process $\varLambda_\epsilon^*:[0,T)\times\Omega\rightarrow\mathbb{R}^{m}$ such that
	\begin{equation*}
		\lim\limits_{\epsilon\rightarrow 0}\mathbb{E}\int_{0}^{T^{'}}\Big|\varLambda_\epsilon(t)-\varLambda_\epsilon^*(t)\Big|^2dt=0,\quad\forall\, 0<T^{'}<T.
	\end{equation*}
\end{proposition}

\begin{proof}
	Applying Theorem \ref{theorem_4} and taking $s=0$, the collection $\{u_\epsilon(\cdot)\}_{\epsilon>0}$ constructed in Lemma \ref{lemma 5} is Cauchy in $L^2_\mathbb{G}(0,T;\mathbb{R}^{m})$. That is
	\begin{equation*}
		\mathbb{E}\int_{0}^{T}\Big|u_{\epsilon_1}(t)-u_{\epsilon_2}(t)\Big|^2dt\rightarrow0,\quad\text{as}\,\,\epsilon_1,\epsilon_2\rightarrow0.
	\end{equation*}
Then, it follows from Lemma \ref{lemma 1} and the linearity of system (\ref{eq400}) that
\begin{equation}\label{eq53}
	\mathbb{E}\left[\sup\limits_{0\leq t\leq T}\Big|\widehat{x}_{\epsilon_1}(t)-\widehat{x}_{\epsilon_2}(t)\Big|^2\right]\leq L\mathbb{E}\int_{0}^{T}\Big|u_{\epsilon_1}(t)-u_{\epsilon_2}(t)\Big|^2dt\rightarrow0 \quad\text{as}\,\,\epsilon_1,\epsilon_2\rightarrow0.
\end{equation}

By the open-loop solvability of Problem (P), Lemma \ref{lemma 7} and Proposition \ref{proposition 2} show that $\{\varTheta_\epsilon(\cdot)\}_{\epsilon>0}$ is Cauchy in $L^2(0,T^{'};\mathbb{R}^{m\times n})$ for any $0< T^{'}<T$. Combining this with (\ref{eq53}), we derive 
\begin{equation*}
	\begin{split}
		&\mathbb{E}\int_{0}^{T^{'}}\Big|\varTheta_{\epsilon_1}(t)\widehat{x}_{\epsilon_1}(t)-\varTheta_{\epsilon_2}(t)\widehat{x}_{\epsilon_2}(t)\Big|^2dt\\
		\leq\ &2\mathbb{E}\int_{0}^{T^{'}}\Big|\varTheta_{\epsilon_1}(t)-\varTheta_{\epsilon_2}(t)\Big|^2\Big|\widehat{x}_{\epsilon_1}(t)\Big|^2dt+2\mathbb{E}\int_{0}^{T^{'}}\Big|\varTheta_{\epsilon_2}(t)\Big|^2\Big|\widehat{x}_{\epsilon_1}(t)-\widehat{x}_{\epsilon_2}(t)\Big|^2dt\\
		\leq\ &2\mathbb{E}\int_{0}^{T^{'}}\Big|\varTheta_{\epsilon_1}(t)-\varTheta_{\epsilon_2}(t)\Big|^2dt\cdot\mathbb{E}\left[\sup\limits_{0\leq t\leq T^{'}}\Big|\widehat{x}_{\epsilon_1}(t)\Big|^2\right]\\
		&+2\mathbb{E}\int_{0}^{T^{'}}\Big|\varTheta_{\epsilon_2}(t)\Big|^2dt\cdot\mathbb{E}\left[\sup\limits_{0\leq t\leq T^{'}}\Big|\widehat{x}_{\epsilon_1}(t)-\widehat{x}_{\epsilon_2}(t)\Big|^2\right]\\
		\rightarrow& \,\,0\  \text{as}\,\,\epsilon_1,\epsilon_2\rightarrow0.
	\end{split}
\end{equation*}
Hence, we have
\begin{equation*}
	\begin{split}
		&\mathbb{E}\int_{0}^{T^{'}}\Big|\varLambda_{\epsilon_1}(t)-\varLambda_{\epsilon_2}(t)\Big|^2dt\\
		=\
		&\mathbb{E}\int_{0}^{T^{'}}\Big|\big[u_{\epsilon_1}(t)-\varTheta_{\epsilon_1}(t)\widehat{x}_{\epsilon_1}(t)\big]-\big[u_{\epsilon_2}(t)-\varTheta_{\epsilon_2}(t)\widehat{x}_{\epsilon_2}(t)\big]\Big|^2dt\\
		\leq\ &2\mathbb{E}\int_{0}^{T^{'}}\Big|u_{\epsilon_1}(t)-u_{\epsilon_2}(t)\Big|^2dt+2\mathbb{E}\int_{0}^{T^{'}}\Big|\varTheta_{\epsilon_1}(t)\widehat{x}_{\epsilon_1}(t)-\varTheta_{\epsilon_2}(t)\widehat{x}_{\epsilon_2}(t)\big]\Big|^2dt\\
		\rightarrow& \,\,0\  \text{as}\,\,\epsilon_1,\epsilon_2\rightarrow0,
	\end{split}
\end{equation*}
which claims the statements of this proposition.
\end{proof}

Finally, we present the main results of this subsection, which proves the equivalence between weak closed-loop solvability and open-loop solvability of Problem (P).

\begin{theorem}\label{theorem 5}
Let	Assumptions \ref{assum_system}, \ref{assum_cost}, \ref{ass_cost>0} hold. When Problem (P) is open-loop solvable, then the limit pair $(\varTheta_\epsilon^*(\cdot),\varLambda_\epsilon^*(\cdot))$ established in Propositions \ref{proposition 2} and \ref{proposition 3} is a weak closed-loop optimal strategy of Problem (P) on any $[s,T)$. In consequence, the  weak closed-loop  and open-loop solvabilities of Problem (P) are equivalent.
\end{theorem}

\begin{proof}
	Given any $(s,x_0)\in[0,T)\times\mathbb{R}^n$, let $\{u_\epsilon(t),s\leq t\leq T\}_{\epsilon>0}$ be the collection obtained in Lemma \ref{lemma 5}. By the open-loop solvability of Problem (P) and Theorem \ref{theorem_4}, we know $\{u_\epsilon(t),s\leq t\leq T\}_{\epsilon>0}$ converges strongly to an open-loop optimal control $\{u^*(t),s\leq t\leq T\}$ of Problem (P). Let $\widehat{x}^*(\cdot)$ be the solution to 
	\begin{equation*}
		\begin{cases}
			\begin{split}
				d\widehat{x}^*(t)= \big[A\widehat{x}^*(t)+Bu^*(t)+b\big]dt
				+\big[C_2\widehat{x}^*(t)+D_2u^*(t)+\sigma_2\big]dW_2(t),\quad t\in[s,T],
			\end{split}\\
			\widehat{x}^*(s)=x_0.\\
		\end{cases}
	\end{equation*}

If we can prove 
\begin{equation}\label{eq54}
	u^*(t)=\varTheta_\epsilon^*(t)\widehat{x}^*(t)+\varLambda_\epsilon^*(t),\quad t\in[s,T),
\end{equation}
then it is evident that $(\varTheta_\epsilon^*(\cdot),\varLambda_\epsilon^*(\cdot))$ is a weak closed-loop optimal strategy of Problem (P) on $[s,T)$. To prove (\ref{eq54}), first we notice that Lemma \ref{lemma 1} implies 
\begin{equation}\label{eq55}
	\mathbb{E}\left[\sup\limits_{0\leq t\leq T}\Big|\widehat{x}_{\epsilon}(t)-\widehat{x}^*(t)\Big|^2\right]\leq L\mathbb{E}\int_{0}^{T}\Big|u_{\epsilon}(t)-u^*(t)\Big|^2dt\rightarrow0 \quad\text{as}\,\,\epsilon\rightarrow0,
\end{equation}
where $\widehat{x}_\epsilon(\cdot)$ evolves according to (\ref{eq400}). 

Furthermore, it follows from Propositions \ref{proposition 2} and \ref{proposition 3} that 
\begin{equation}\label{eq56}
	\begin{split}
		&\lim\limits_{\epsilon\rightarrow 0}\int_{0}^{T^{'}}\Big|\varTheta_\epsilon(t)-\varTheta_\epsilon^*(t)\Big|^2dt=0,\quad\forall\, 0<T^{'}<T,\\
		&\lim\limits_{\epsilon\rightarrow 0}\mathbb{E}\int_{0}^{T^{'}}\Big|\varLambda_\epsilon(t)-\varLambda_\epsilon^*(t)\Big|^2dt=0,\quad\forall\, 0<T^{'}<T.
	\end{split}
\end{equation}
Combining (\ref{eq55}) with (\ref{eq56}), we derive
\begin{equation*}
	\begin{split}
		&\mathbb{E}\int_{0}^{T^{'}}\Big|u_{\epsilon}(t)-\big[\varTheta_\epsilon^*(t)\widehat{x}^*(t)+\varLambda_\epsilon^*(t)\big]\Big|^2dt\\
		=\
		&\mathbb{E}\int_{0}^{T^{'}}\Big|\big[\varTheta_{\epsilon}(t)\widehat{x}_{\epsilon}(t)+\varLambda_\epsilon(t)\big]-\big[\varTheta_\epsilon^*(t)\widehat{x}^*(t)+\varLambda_\epsilon^*(t)\big]\Big|^2dt\\
		\leq\ &2\mathbb{E}\int_{0}^{T^{'}}\Big|\varLambda_\epsilon(t)-\varLambda_\epsilon^*(t)\Big|^2dt+4\mathbb{E}\int_{0}^{T^{'}}\Big|\varTheta_{\epsilon}(t)-\varTheta_{\epsilon}^*(t)\Big|^2\Big|\widehat{x}_\epsilon(t)\Big|^2dt\\
		&+4\mathbb{E}\int_{0}^{T^{'}}\Big|\varTheta_{\epsilon}^*(t)\Big|^2\Big|\widehat{x}_{\epsilon}(t)-\widehat{x}^*(t)\Big|^2dt\\
		\leq\ &2\mathbb{E}\int_{0}^{T^{'}}\Big|\varLambda_\epsilon(t)-\varLambda_\epsilon^*(t)\Big|^2dt+4\mathbb{E}\int_{0}^{T^{'}}\Big|\varTheta_{\epsilon}(t)-\varTheta_{\epsilon}^*(t)\Big|^2dt\cdot\mathbb{E}\left[\sup\limits_{0\leq t\leq T^{'}}\Big|\widehat{x}_{\epsilon}(t)\Big|^2\right]\\
		&+4\mathbb{E}\int_{0}^{T^{'}}\Big|\varTheta_{\epsilon}^*(t)\Big|^2dt\cdot\mathbb{E}\left[\sup\limits_{0\leq t\leq T^{'}}\Big|\widehat{x}_{\epsilon}(t)-\widehat{x}^*(t)\Big|^2\right]\\
		\rightarrow& \,\,0\  \,\,\text{as}\,\,\epsilon\rightarrow0.
	\end{split}
\end{equation*}
Since $\{u_\epsilon(t),s\leq t\leq T\}_{\epsilon>0}$ converges strongly to $\{u^*(t),s\leq t\leq T\}$ in $L^2_\mathbb{G}(s,T;\mathbb{R}^m)$ as $\epsilon\rightarrow0$, (\ref{eq54}) holds true. The above analysis means that the open-loop solvability of Problem (P) implies the weak closed-loop solvability of Problem (P). The reverse argument is clear with the help of Definitions \ref{defnition 1} and \ref{definition 3}. The proof is therefore completed.
\end{proof}

\section{An illustrative example}\label{sec 5}
In this section, we present an example to demonstrate the procedures for obtaining weak closed-loop strategies. This section considers the following one-dimensional system
\begin{equation*}
	\begin{cases}
		\begin{split}
			dx(t)= \,\,&\big[-x(t)+u(t)+b(t)\big]dt+dW_1(t)
			+\sqrt{2}x(t)dW_2(t),\quad t\in[s,1],\\
		\end{split}\\
		x(s)=x_0,
	\end{cases}
\end{equation*}
and cost functional 
\begin{equation*}
	\begin{split}
		J(s,x_0;u(\cdot))=\ &\mathbb{E}\left[ x(T)^\top x(T)\right],
	\end{split}
\end{equation*}
where $b(\cdot)$ and $\sigma_1(\cdot)$ are 
\begin{equation*}
	b(t)=\begin{cases}
		\frac{e^{\sqrt{2}W_2(t)-2t}}{\sqrt{1-s}},\quad t\in[0,1),\\
		0,\quad\quad\quad\quad\quad t=1.
	\end{cases}
\end{equation*}
As stated in \cite[Section 5]{WangSunYong2019}, we know $b(\cdot)\in L^2_{\mathbb{G}}(\Omega;L^1([0,1];\mathbb{R}^1))$.

In the sequel, we prove the open-loop solvability of above problem by applying Theorem \ref{theorem_4}, and,  consequently, by Theorem \ref{theorem 5}, this problem is weak closed-loop solvable. Without loss of generality, we let $s=0$. In this example, Riccati equation (\ref{Riccati_2_epsilon}) reads 
\begin{equation*}
	\begin{cases}
		\dot{P}_\epsilon^2-\frac{\big(P_\epsilon^2(t)\big)^\top P_\epsilon^2(t)}{\epsilon}=0,\\
		P_\epsilon^2(1)=1,
	\end{cases}
\end{equation*}
and its solution is
\begin{equation*}
	P_\epsilon^2(t)=\frac{\epsilon}{\epsilon+1-t},\quad t\in[0,1].
\end{equation*}
Then we know 
\begin{equation}\label{eq58}
	\begin{split}
		\varTheta_\epsilon&\triangleq-(R+\epsilon\mathbb{I}_m+D_1^\top P^1 D_1+D_2^\top P^2_\epsilon D_2)^{-1}(B^\top P^2_\epsilon+D_1^\top P^1 C_1+D_2^\top P^2_\epsilon C_2+S)\\
		&=-\frac{P_\epsilon^2}{\epsilon}=-\frac{1}{\epsilon+1-t},\quad t\in[0,1],
	\end{split}
\end{equation}
and BSDE (\ref{BSDE_epsilon}) becomes
\begin{equation*}
	\begin{cases}
		\begin{split}
			d\alpha_\epsilon(t)=&-\Big[\big(-1+\varTheta_\epsilon\big)^\top\alpha_\epsilon(t)
			+\sqrt{2}\beta_\epsilon(t)+P^2_\epsilon b\Big]dt+\beta_\epsilon(t) dW_2(t),\\
		\end{split}\\
		\alpha_\epsilon(1)=0.\\
	\end{cases}
\end{equation*}
By defining $\eta(t)=\frac{1}{\sqrt{1-t}}$ and using the variation of constants formula for BSDEs, we derive
\begin{equation*}
	\begin{split}	\alpha_\epsilon(t)=&\frac{\epsilon}{\epsilon+1-t}e^{2t-\sqrt{2}W_2(t)}\mathbb{E}\Big[\int_{t}^{1}e^{\sqrt{2}W_2(r)-2r}b(r)dr\Big|\mathcal{G}_t\Big]\\
		=&\frac{\epsilon}{\epsilon+1-t}e^{2t-\sqrt{2}W_2(t)}\mathbb{E}\Big[\int_{t}^{1}e^{2\sqrt{2}W_2(r)-4r}\eta(r)dr\Big|\mathcal{G}_t\Big]\\
		=&\frac{\epsilon}{\epsilon+1-t}e^{\sqrt{2}W_2(t)-2t}\int_{t}^{1}\eta(r)dr.
	\end{split}
\end{equation*}
Let
\begin{equation}\label{eq59}
	\begin{split}
		\varLambda_\epsilon\triangleq&-(R+\epsilon\mathbb{I}_m+D_1^\top P^1 D_1+D_2^\top P^2_\epsilon D_2)^{-1}(B^\top\alpha_\epsilon+D_2^\top\beta_\epsilon+D_1^\top P^1\sigma_1+D_2^\top P^2_\epsilon\sigma_2+\rho)\\
		=&-\frac{\alpha_\epsilon}{\epsilon}=-\frac{1}{\epsilon+1-t}e^{\sqrt{2}W_2(t)-2t}\int_{t}^{1}\eta(r)dr,\quad t\in[0,1].\\
	\end{split}
\end{equation}
Then closed-loop system (\ref{eq400}) becomes
	\begin{equation*}
	\begin{cases}
		\begin{split}
			d\widehat{x}_\epsilon(t)= \big[[\varTheta_\epsilon(t)-1]\widehat{x}_\epsilon(t)+\varLambda_\epsilon+b\big]dt
			+\sqrt{2}\widehat{x}_\epsilon(t)dW_2(t),\quad t\in[0,1],
		\end{split}\\
		\widehat{x}_\epsilon(0)=x_0.\\
	\end{cases}
\end{equation*}
According to the variation of constants formula for SDEs, we obtain
\begin{equation*}\label{eq60}
	\begin{split}
		\widehat{x}_\epsilon(t)=&(\epsilon+1-t)e^{\sqrt{2}W_2(t)-2t}\int_{0}^{t}\frac{1}{\epsilon+1-r}e^{-\sqrt{2}W_2(r)+2r}\big[\varLambda_\epsilon(r)+b(r)\big]dr\\
		&+\frac{\epsilon+1-t}{\epsilon+1}e^{\sqrt{2}W_2(t)-2t}x_0,\quad t\in[0,1].
	\end{split}
\end{equation*}

To show the open-loop solvability of this example at $(0,x_0)$, Theorem \ref{theorem_4} implies that it is enough to prove that the following collection 
\begin{equation}\label{eq61}
	\begin{split}
		u_\epsilon(t)\triangleq\ &\varTheta_\epsilon(t)\widehat{x}_\epsilon(t)+\varLambda_\epsilon(t)\\
		=\ &-e^{\sqrt{2}W_2(t)-2t}\int_{0}^{t}\frac{1}{\epsilon+1-r}e^{-\sqrt{2}W_2(r)+2r}\big[\varLambda_\epsilon(r)+b(r)\big]dr\\
		&-\frac{1}{\epsilon+1}e^{\sqrt{2}W_2(t)-2t}x_0+\varLambda_\epsilon(t),\quad t\in[0,1],
	\end{split}
\end{equation} 
is bounded in $L^2_\mathbb{G}(0,1;\mathbb{R}^1)$. To this end, we analyze some terms in this equation. According to Fubini's theorem, we have
\begin{equation}\label{eq62}
	\begin{split}
		&\int_{0}^{t}\frac{1}{\epsilon+1-r}e^{-\sqrt{2}W_2(r)+2r}\varLambda_\epsilon(r)dr\\
		=\ &-\int_{0}^{t}\frac{1}{(\epsilon+1-r)^2}\int_{r}^{1}\eta(\tau)d\tau dr\\
		=\ &-\int_{0}^{t}\eta(\tau)\int_{0}^{\tau}\frac{1}{(\epsilon+1-r)^2} drd\tau-\int_{t}^{1}\eta(\tau)\int_{0}^{t}\frac{1}{(\epsilon+1-r)^2}drd\tau\\
		=\ &-\int_{0}^{t}\frac{1}{\epsilon+1-\tau}\eta(\tau)d\tau+\frac{1}{\epsilon+1}\int_{0}^{1}\eta(\tau)d\tau-\frac{1}{\epsilon+1-t}\int_{t}^{1}\eta(\tau)d\tau.
	\end{split}
\end{equation}
Moreover, 
\begin{equation}\label{eq63}
	\begin{split}
		\int_{0}^{t}\frac{1}{\epsilon+1-r}e^{-\sqrt{2}W_2(r)+2r}b(r)dr
		=\int_{0}^{t}\frac{1}{\epsilon+1-r}\eta(r)dr.
	\end{split}
\end{equation}
Combining (\ref{eq59}), (\ref{eq61}), (\ref{eq62}) with (\ref{eq63}), we derive
\begin{equation}\label{eq64}
	\begin{split}
		u_\epsilon(t)=-\frac{1}{\epsilon+1}e^{\sqrt{2}W_2(t)-2t}x_0-e^{\sqrt{2}W_2(t)-2t}\frac{1}{\epsilon+1}\int_{0}^{1}\eta(r)dr=-\frac{2+x_0}{\epsilon+1}e^{\sqrt{2}W_2(t)-2t},\quad t\in[0,1].
	\end{split}
\end{equation}
Consequently, we know
\begin{equation*}
	\mathbb{E}\int_{0}^{1}\big|u_\epsilon(t)\big|^2dt=\left(\frac{2+x_0}{\epsilon+1}\right)^2\leq(2+x_0)^2,
\end{equation*}
which shows the boundedness of $\{u_\epsilon(\cdot)\}_{\epsilon>0}$ in $L^2_\mathbb{G}(0,1;\mathbb{R}^1)$. By letting $\epsilon\rightarrow 0$, we obtain an open-loop optimal control of Problem (P), which is 
\begin{equation*}
	\begin{split}
		u^*(t)=-(2+x_0)e^{\sqrt{2}W_2(t)-2t},\quad t\in[0,1].
	\end{split}
\end{equation*}
Last but not the least, by letting $\epsilon\rightarrow 0$ in (\ref{eq58}) and (\ref{eq59}), we get a weak closed-loop optimal strategy
\begin{equation*}
	\begin{split}
		&\varTheta_\epsilon^*(t)=\lim\limits_{\epsilon\rightarrow 0}\varTheta_\epsilon(t)=-\frac{1}{1-t}, \quad t\in[0,1),\\
		&\varLambda_\epsilon^*(t)=\lim\limits_{\epsilon\rightarrow 0}\varLambda_\epsilon(t)=-\frac{1}{1-t}e^{\sqrt{2}W_2(t)-2t}\int_{t}^{1}\eta(r)dr=\frac{-2e^{\sqrt{2}W_2(t)-2t}}{\sqrt{1-t}},\quad t\in[0,1).
	\end{split}
\end{equation*}
In addition, since
\begin{equation*}
	\begin{split}
		&\mathbb{E}\int_{0}^{1}\big|\varTheta_\epsilon^*(t)\big|^2dt=\int_{0}^{1}\frac{1}{(1-t)^2}dt=\infty\\
		&\mathbb{E}\int_{0}^{1}\big|\varLambda_\epsilon^*(t)\big|^2dt=\mathbb{E}\int_{0}^{1}\frac{4e^{2\sqrt{2}W_2(t)-4t}}{\sqrt{1-t}}dt=\int_{0}^{1}\frac{4}{1-s}ds=\infty,
	\end{split}
\end{equation*}
 $\varTheta_\epsilon^*(\cdot)$ and $\varLambda_\epsilon^*(\cdot)$ are not square-integrable on $[0,1)$.
 
\section{Conclusions and  future works}\label{sec 6}
This paper considered an LQSOC problem with partial information. The open-loop, closed-loop and weak closed-loop solvabilities of this problem were introduced and investigated.  The equivalence between open-loop and weak solvabilities was proved, and finally, an example is given to demonstrate the procedures for finding the optimal weak closed-loop strategy. Inspired by the results presented in this paper, there are also some other works worth studying. For example, it is of great interest to investigate open-loop and weak closed-loop solvabilities for LQSOC problems with partial observation. That is, the information available to the agent is not a Brownian motion, but a noisy observation of the state.  To study this issue, we may need to exploit more filtering theories. We hope to make some progress in this direction in the future.

\section*{Acknowledgements}
\noindent  The last author is very grateful for  many useful and inspiring discussions with Ms. Siqi Feng. 

\section*{Statements and Declarations}
\noindent \textbf{Funding} X. Li acknowledges the financial support by the Hong Kong General Research Fund, China, under Grant Nos. 15216720, 15221621 and 15226922. G. Wang acknowledges the financial support from  the National Natural Science Foundation of China under Grant Nos. 61925306, the National Key R\&D Program of China under Grant No. 2022YFA1006103, and the Natural Science Foundation of Shandong Province under Grant No. ZR2019ZD42. J. Xiong acknowledges the financial support from the National Key R\&D Program of China under Grant No. 2022YFA1006102 and the National Natural Science Foundation of China under Grant No. 12471418.

\noindent \textbf{Conflict of interest} The authors declare that they have no conflict of interest.


\begin{thebibliography}{99}
\bibitem{Bensoussan}
A. Bensoussan, Lectures on stochstic control, part I, in Nonlinear Filtering and Stochastic Control, Lecture Notes in Math., Springer-Verlag, Berlin, 972, 1–62 (1982).

\bibitem{Letov}
A. M. Letov, The analytical design of control systems, Automat. Remote Control, 22, 363-372 (1961).

\bibitem{Baghery}
F. Baghery, B. Øksendal, A maximum principle for stochastic control with partial information, Stoch. Anal. Appl., 25, 705-717 (2007). 

\bibitem{WangXiao2015}
G. C. Wang, H. Xiao, Arrow sufficient conditions for optimality of fully coupled forward–backward stochastic differential
equations with applications to finance, J. Optim. Theory Appl.,  165, 639-656 (2015).

\bibitem{WangWuXiongBook2018}
G. C. Wang, Z. Wu, J. Xiong, An Introduction to Optimal Control of FBSDE with Incomplete Information, Springer, Cham, Switzerland, 2018.

\bibitem{WangSunYong2019}
H. X. Wang, J. R. Sun, J. M. Yong,  Weak closed-loop solvability of stochastic linear-quadratic optimal control problems, Disc. Cont. Dyn. Syst. A 39, 2785-2805 (2019).

\bibitem{HuangWangXiong2009}
J. H. Huang, G. C. Wang, J. Xiong, A maximum principle for partial information backward stochastic control problems with applications, SIAM J. Control Optim., 48, 2106-2117 (2009).

\bibitem{Bismut}
J. M. Bismut, Linear quadratic optimal stochastic control with random coefficients, SIAM J. Control Optim., 14, 419–444 (1976).


\bibitem{YongZhou1999}
J. M. Yong, X. Y. Zhou, Stochastic Controls: Hamiltonian Systems and HJB Equations, Springer-Verlag, NY, 1999.	

\bibitem{WenLiXiong2021}
J. Q. Wen, X. Li, J. Xiong, Weak closed-loop solvability of stochastic linear quadratic optimal control problems of markovian regime switching system,  Appl. Math. Optim., 84, 535-565 (2021).

\bibitem{Sun_mean-field_open}
J. R. Sun,  Mean-field stochastic linear quadratic optimal control problems: open-loop solvabilities, ESAIM:COCV, 23, 1099-1127 (2017).

\bibitem{SunWang_weak_closed}
J. R. Sun,  H. X. Wang, Mean-field stochastic linear-quadratic optimal control problems: Weak closed-loop solvability, Math. Control Relat. Fields, 11, 17-71 (2021).

\bibitem{SunYong2014_Game}
J. R. Sun, J. M. Yong, Linear quadratic stochastic differential games: Open-loop and closed-loop saddle points, SIAM J. Control Optim., 52, 4082–4121 (2014).

\bibitem{SunLiYong2016_SLQ}
J. R. Sun, X. Li, J. M. Yong, Open-loop and closed-loop solvabilities for stochastic linear quadratic optimal control problems, SIAM J. Control Optim., 54, 2274–2308 (2016).

\bibitem{AitRamiChen2001}
M. Ait Rami, X. Chen,  J. B. Moore, X. Y. Zhou,  Solvability and asymptotic behavior of generalized Riccati equations arising in indefinite stochastic LQ controls, IEEE Trans. Autom. Control  46,  428-440 (2001).

\bibitem{AitRamiZhou2000}
M. Ait Rami,  X. Y. Zhou, Linear matrix inequalities, Riccati equations, and indefinite stochastic linear quadratic control, IEEE Trans. Autom. Control 45(6), 1131-1143 (2000).

\bibitem{Meng2009}
Q. X. Meng,  A maximum principle for optimal control problem of fully coupled forward-backward stochastic systems with partial information, Sci. China Ser. A-Math.,  52, 1579-1588 (2009). 



\bibitem{Bellman}
R. Bellman, I. Glicksberg, O. Gross, Some Aspects of the Mathematical Theory of Control Processes, RAND Corporation, Santa Monica, CA, 1958.

\bibitem{Kalman}
R. E. Kalman, Contributions to the theory of optimal control, Bol. Soc., Mat. Mexicana, 5, 102-119 (1960).

\bibitem{ChenZhou1998}
 S. P. Chen, X. J. Li, X. Y. Zhou,  Stochastic linear quadratic regulators with indefinite control weight costs, SIAM J. Control Optim., 36, 1685-1702 (1998).

 
\bibitem{ChenYong2001}
S. P. Chen, J. M. Yong, Stochastic linear quadratic optimal control problems, Appl. Math. Optim., 43, 21–45, (2001).

\bibitem{ChenZhou2000}
S. P. Chen, X. Y. Zhou, Stochastic linear quadratic regulators with indefinite control weight costs. II, SIAM J. Control Optim., 39, 1065–1081 (2000).


\bibitem{Wonham}
W. M. Wonham, On a matrix Riccati equation of stochastic control, SIAM J. Control, 6, 681-697 (1968).


\bibitem{LiSunYong_mean-field_closed}
X. Li, J. R. Sun, J. M. Yong, Mean-field stochastic linear quadratic optimal control problems: closed-loop solvability, Probab. Uncertain. Qua., 1, 2 (2016).


\bibitem{ZhangLiXiong2021}
X. Zhang, X. Li, J. Xiong, Open-loop and closed-loop solvabilities for stochastic linear quadratic optimal control problems of Markov regime-switching system, ESAIM: COCV 27, 69 (2021).




	
	
	
	
	
	
	
	
	














\end{thebibliography}
\end{document}